\documentclass[11pt]{amsart}
\usepackage{amsmath,amssymb,amsthm,amsxtra} 
\usepackage{stmaryrd}
\usepackage{times}
\usepackage[margin=1.5in]{geometry}
\usepackage[linktocpage=true, hidelinks, colorlinks=true, allcolors=blue]{hyperref}
\usepackage{comment}

\makeatletter
\def\l@subsection{\@tocline{2}{0pt}{2.5pc}{5pc}{}}
\renewcommand\tocchapter[3]{%
  \indentlabel{\@ifnotempty{#2}{\ignorespaces#2.\quad}}#3%
}
\newcommand\@dotsep{4.5}
\def\@tocline#1#2#3#4#5#6#7{\relax
  \ifnum #1>\c@tocdepth 
  \else
    \par \addpenalty\@secpenalty\addvspace{#2}%
    \begingroup \hyphenpenalty\@M
    \@ifempty{#4}{%
      \@tempdima\csname r@tocindent\number#1\endcsname\relax
    }{%
      \@tempdima#4\relax
    }%
    \parindent\z@ \leftskip#3\relax \advance\leftskip\@tempdima\relax
    \rightskip\@pnumwidth plus1em \parfillskip-\@pnumwidth
    #5\leavevmode\hskip-\@tempdima{#6}\nobreak
    \leaders\hbox{$\m@th\mkern \@dotsep mu\hbox{.}\mkern \@dotsep mu$}\hfill
    \nobreak
    \hbox to\@pnumwidth{\@tocpagenum{#7}}\par
    \nobreak
    \endgroup
  \fi}
\makeatother
\AtBeginDocument{%
\makeatletter
\expandafter\renewcommand\csname r@tocindent0\endcsname{0pt}
\makeatother
}
\def\l@subsection{\@tocline{2}{0pt}{2.5pc}{5pc}{}}

\newtheorem{thm}{Theorem}[section]
\newtheorem{lemma}[thm]{Lemma}
\newtheorem{proposition}[thm]{Proposition}

\newtheorem{definition}[thm]{Definition}
\newtheorem{corollary}[thm]{Corollary}

\newtheorem{question}[thm]{Question}

\newtheorem*{maintheorem*}{Main Theorem}
\newtheorem*{theorem*}{Theorem}
\newtheorem*{corollary*}{Corollary}

\newcommand{\p}{\mathbb{P}}
\newcommand{\q}{\mathbb{Q}}

\newcommand{\con}{\ \widehat{} \ }

\newcommand{\dom}{\mathrm{dom}}

\newcommand{\h}{\mathrm{ht}}

\newcommand{\res}{\upharpoonright}

\newcommand{\Lim}{\text{Lim}}

\newcommand{\lem}{(\le \! m)}
\newcommand{\len}{(\le \! n)}
\newcommand{\les}{(\le \! s)}
\newcommand{\Imm}{\mathrm{Imm}}
\newcommand{\up}{{\uparrow}}

\begin{document}

\title[Some Results on Finitely Splitting Subtrees of Aronszajn Trees]{Some Results on Finitely Splitting Subtrees of \\ Aronszajn Trees}

\author{John Krueger}

\address{John Krueger, Department of Mathematics, 
	University of North Texas,
	1155 Union Circle \#311430,
	Denton, TX 76203, USA}
\email{john.krueger@unt.edu}

\date{August 18, 2024; revised June 7, 2025}



\begin{abstract}
	For any $2 \le n < \omega$, 
	we introduce a forcing poset using generalized promises 
	which adds a normal $n$-splitting subtree to a 
	$(\ge \! n)$-splitting normal Aronszajn tree. 
	Using this forcing poset, we prove several consistency results concerning 
	finitely splitting subtrees of Aronszajn trees. 
	For example, it is consistent that there exists an infinitely splitting Suslin tree whose 
	topological square is not Lindel\"{o}f, which solves an open problem due to Marun. 
	For any $2 < n < \omega$, it is consistent that every $(\ge \! n)$-splitting normal 
	Aronszajn tree contains a normal $n$-splitting subtree, but there exists a normal 
	infinitely splitting Aronszajn tree which contains no $(< \! n)$-splitting subtree. 
	To show the latter consistency result, we prove a forcing iteration preservation 
	theorem related to not adding new small-splitting subtrees of Aronszajn trees.	
\end{abstract}

\maketitle

\tableofcontents

\section{Introduction}

In this article we study finitely splitting subtrees of Aronszajn trees. 
We prove consistency results related to the issue of the existence or non-existence 
of uncountable downwards closed subtrees of Aronszajn trees whose elements all have 
some prescribed number of immediate successors. 
This topic has deep connections with fundamental notions in the theory of trees, 
such as the property of being Suslin, as well as having significance to topology. 
For example, recently Marun \cite{marun} has shown that an $\omega_1$-tree, 
equipped with a natural topology called the 
fine wedge topology, 
is Lindel\"{o}f if and only it does not contain an 
uncountable downwards closed finitely splitting subtree. 

We introduce a forcing poset for adding an uncountable downwards closed normal subtree of a 
given Aronszajn tree with some specified arity of its elements using the 
technique of promises. 
The idea of a promise was introduced by Shelah 
\cite[Chapter V, Section 6]{properimproper} and incorporated in his forcing for 
specializing an Aronzajn tree without adding reals. 
Roughly speaking, a promise is an uncountable 
subtree of a derived tree of an Aronszajn tree which can be used in a forcing 
notion to put constraints on the growth of the working part of a condition. 
While introduced explicitly by Shelah, a version of promises appeared earlier 
within Jensen's proof of the consistency of Suslin's hypothesis 
with the continuum hypothesis (\cite[page 108]{devlinj}). 
More recently, the use of promises to force subtrees of Aronszajn trees appears in work of 
Lamei Ramandi and Moore (\cite{mooreminimal}, \cite{mooreramandi}). 
Our forcing for adding a subtree 
uses a generalized version of promises which is similar to that introduced by 
Abraham-Shelah \cite[Definition 4.1]{AS2} to specialize an Aronszajn tree while 
preserving Suslin trees.

We now summarize the main results of the article. 
Let $2 \le n < \omega$ and assume that $T$ is a normal Aronszajn tree such that every element 
of $T$ has at least $n$-many immediate successors. 
We prove that there exists a forcing poset which adds an uncountable downwards closed 
normal subtree of $T$ such that every element of the subtree 
has exactly $n$-many immediate successors in the subtree. 
This forcing poset has several strong properties: it is totally proper, 
$\omega_2$-c.c.\ assuming \textsf{CH}, and can be iterated with countable support 
without adding reals. 
The forcing preserves Aronszajn trees, and assuming that $T$ does not contain a Suslin subtree, 
it preserves Suslin trees. 
In addition, for any $2 \le m \le n$, 
if $S$ is any Aronszajn tree which does not contain an uncountable downwards closed 
subtree whose elements have fewer than $m$-many immediate successors in the subtree, then 
our forcing preserves this property of $S$.

Using this forcing poset we prove several consistency results, two of 
which we highlight here. 
Solving an open problem due to Marun, 
we show that it is consistent that there exists an infinitely splitting normal $\omega_1$-tree 
$S$ which is Suslin, and hence Lindel\"{o}f in the fine wedge topology, but the 
topological square $S \times S$ is not Lindel\"{o}f. 
We prove that for all $2 < n < \omega$, it is consistent that any normal Aronszajn tree 
satisfying that every element has at least $n$-many immediate successors contains an 
uncountable downwards closed normal $n$-ary subtree, yet there exists a normal infinitely splitting 
Aronszajn tree which contains no uncountable downwards closed subtree for which every 
element has fewer than $n$-many immediate successors. 
In order to show the last consistency result, we prove a new preservation theorem 
which say that any countable support forcing iteration of forcings which are proper and 
do not add small-splitting subtrees of Aronszajn trees also does not add 
such subtrees.

\bigskip

\emph{Background.} We assume that the reader is familiar with $\omega_1$-trees, proper forcing, 
and iterated forcing. 
We describe our notation and some basic results about trees which we need.

Let $T$ be an $\omega_1$-tree, which means a tree of height $\omega_1$ with 
countable levels. 
For any $0 < n \le \omega$, 
$T$ is \emph{$n$-splitting} if every element of $T$ 
has exactly $n$-many immediate successors. 
Similar properties such as $\len$-splitting, $(< \! n)$-splitting, and $(\ge \! n)$-splitting  
(when $n$ is finite) have the obvious meaning. 
By a \emph{subtree} of $T$ 
we mean any set $W \subseteq T$ considered as a tree with the 
order $<_T \cap \ W^2$. 
We will mostly, but not exclusively, 
be interested in subtrees which are downwards closed in the obvious sense. 
For any $x \in T$, let $T_x$ be the subtree $\{ y \in T : x \le_T y \}$. 

For any $x \in T$, we write $\h_T(x)$ for the height of $x$ in $T$. 
Let $T_\alpha$, or \emph{level $\alpha$ of $T$}, denote the set of all elements 
of $T$ with height $\alpha$, for any $\alpha < \omega_1$. 
Let $T \res A = \{ x \in T : \h_T(x) \in A \}$, for any 
$A \subseteq \omega_1$. 
If $b$ is a cofinal branch of $T$ and $\alpha < \omega_1$, 
let $b(\alpha)$ be the unique element of $b \cap T_\alpha$. 
Whenever $x \in T$ and $\beta < \h_T(x)$, $x \res \beta$ is the unique 
element below $x$ with height $\beta$. 
Similarly, $X \res \beta = \{ x \res \beta : x \in X \}$ and 
$\vec x \res \beta = (x_0 \res \beta,\ldots,x_{n-1} \res \beta)$ whenever 
$X \subseteq T_\alpha$ and $\vec x = (x_0,\ldots,x_{n-1}) \in T_\alpha^n$ for some 
$\beta < \alpha < \omega_1$. 

For any $x \in T$, we write $\Imm_T(x)$ for the set of immediate successors of $x$. 
We say that $T$ is \emph{normal} if it has a root, is Hausdorff, every element has 
incomparable elements above it, and every element has uncountably many elements 
above it (if $T$ is Aronszajn, then the third assumption is subsumed by the fourth). 
A subset of $T$ is \emph{dense open} if it is dense open in the forcing poset 
consisting of $T$ partially ordered by the reverse of the tree ordering.

A product tree $S_0 \otimes \cdots \otimes S_{n-1}$, where $S_0, \ldots, S_{n-1}$ are 
finitely many $\omega_1$-trees, is the $\omega_1$-tree 
consisting of the tuples in $S_0 \times \cdots \times S_{n-1}$ 
whose elements all have the same height, ordered component-wise. 
Note that if $S_m$ is normal for all $m < n$, then so is the product. 
A \emph{derived tree} of an $\omega_1$-tree $T$ with dimension $n$ is a tree of the form 
$T_{x_0} \otimes \cdots \otimes T_{x_{n-1}}$, where $x_0, \ldots, x_{n-1}$ are distinct 
elements of $T$ with the same height. 
If $\vec x = (x_0,\ldots,x_{n-1})$, we abbreviate 
$T_{x_0} \otimes \cdots \otimes T_{x_{n-1}}$ by $T_{\vec x}$.

The following lemma summarizes some basic facts about trees which we need.

\begin{lemma}
	Suppose that $T$ is an $\omega_1$-tree.
	\begin{enumerate}
	\item[(a)] There exists an uncountable downwards closed subtree $U$ of $T$ 
	such that every element of 
	$U$ has uncountably many elements of $U$ above it.
	\item[(b)] If $T$ is normal, then there exists a club $C \subseteq \omega_1$ such that 
	$T \res C$ is infinitely splitting.
	\item[(c)] If there exists a club $C \subseteq \omega_1$ 
	such that $T \res C$ is special, then $T$ is special.
	\item[(d)] If $T$ is normal, then the following are equivalent:
	\begin{itemize}
	\item[(i)] $T$ is Suslin;
	\item[(ii)] every dense open subset of $T$ contains some level of $T$;
	\item[(iii)] for every uncountable downwards closed set $U \subseteq T$, 
	there exists some $x \in T$ such that $T_x \subseteq U$.
	\end{itemize}
	\end{enumerate}
\end{lemma} 

We say that a cardinal $\lambda$ is \emph{large enough} if all parameters under discussion 
are in $H(\lambda)$. 
A forcing poset is \emph{totally proper} if it is proper and it does not add reals. 
If $N$ is a countable elementary substructure of $H(\lambda)$ and $\q \in N$ 
is a forcing poset, 
a condition $q \in \q$ is a \emph{total master condition for $N$} if for any 
dense open set $D$ of $\q$ in $N$, there exists some $v \in D \cap N$ 
such that $q \le v$. 
We write $\Lim(\omega_1)$ for the set of countable limit ordinals.

\section{Finitely Splitting Subtrees and Lindel\"{o}f Trees}

We begin our study by providing some \textsf{ZFC} constructions of finitely 
splitting Aronszajn trees and Aronszajn trees which contain uncountable downwards 
closed finitely splitting subtrees. 
These results are obtained by starting with a given tree and then modifying it to have 
some desired property. 
We then define the fine wedge topology on an $\omega_1$-tree and 
describe a characterization of the Lindel\"{o}f property 
of this topological space in terms of the non-existence of finitely splitting subtrees.

A standard construction of a special Aronszajn tree is to build a normal $\omega_1$-tree 
$T$ together with a strictly increasing function $\pi : T \to \q$ by recursion, 
maintaining that whenever $\pi(x) < q \in \q$ and $\h_T(x) < \alpha < \omega_1$, 
then there exists some $y \in T_\alpha$ above $x$ such that $\pi(y) < q$. 
Note that this inductive hypothesis implies that $T$ 
is infinitely splitting.

\begin{lemma}
	For every $2 \le n < \omega$, 
	there exists a special Aronszajn tree which is normal and $n$-splitting.
\end{lemma}

\begin{proof}
	Start with a special Aronszajn tree $T$ which is normal and infinitely splitting. 
	We build a normal $\omega_1$-tree $U$ satisfying that 
	$U \res (\Lim(\omega_1) \cup \{ 0 \}) = T$. 
	For any $x \in T$, insert $\omega$-many new levels in between $x$ and its immediate 
	successors in $T$ which is a copy of the tree 
	$({}^{< \omega} n \setminus \{ \emptyset \},\subset)$. 
	Now pick distinct cofinal branches $\langle b_k : k < \omega \rangle$ of this copy 
	of ${}^{< \omega} n \setminus \{ \emptyset \}$ 
	so that every element of the copy belongs to some branch. 
	Enumerating the immediate successors of $x$ in $T$ as $\langle y_k : k < \omega \rangle$, 
	specify $y_k$ to be the unique upper bound of $b_k$ for all $k < \omega$. 
	It is routine to check that the tree $U$ obtained in this manner is normal 
	and $n$-splitting. 
	Since $U$ restricted to the club set of countable limit ordinals is special, 
	by Lemma 1.1(c) so is $U$.
\end{proof}

\begin{proposition}
	There exists a special Aronszajn tree $U$ which is normal, 
	infinitely splitting, and contains a normal uncountable downwards closed 
	$2$-splitting subtree. 
	In fact, for each $x \in U$, $U_x$ contains such a subtree.
\end{proposition}

\begin{proof}
	By Lemma 2.1, we can fix a special Aronszajn tree $T$ which is normal and $2$-splitting. 
	Let $\pi : T \to \omega$ be a specializing function. 
	We construct by induction a sequence of normal $\omega_1$-trees 
	$\langle U^\alpha : \alpha < \omega_1 \rangle$ satisfying that $U^0 = T$ and 
	for all $\alpha < \beta < \omega_1$, $U^\alpha$ is a downwards closed subtree of 
	$U^\beta$ and $U^\beta \res (\alpha+1) = U^\alpha \res (\alpha+1)$. 
	Note that since $U^\alpha$ is downwards closed in $U^\beta$, 
	$\h_{U^\beta} \res U^\alpha = \h_{U^\alpha}$.

	Begin by letting $U^0 = T$. 
	Suppose that $U^\alpha$ is defined for some fixed $\alpha < \omega_1$. 
	Define $U^{\alpha+1}$ which extends $U^\alpha$ by adding above each 
	$x \in U^\alpha_\alpha$ infinitely many pairwise disjoint trees, which we denote by 
	$U(\alpha,x,n)$ for $n < \omega$, 
	each of them isomorphic to $T$ minus the root of $T$. 
	Since $T$ is special, we can fix a specializing function 
	$\pi_{\alpha,x,n} : U(\alpha,x,n) \to \omega$ for each such $x$ and $n$. 
	This completes the constriction of $U^{\alpha+1}$, which easily satisfies the 
	required properties. 
	Now assume that $\delta < \omega_1$ is a limit ordinal and 
	$U^\alpha$ is defined for all $\alpha < \delta$. 
	Let $U^\delta = \bigcup_{\alpha < \delta} U^\alpha$. 
	Again, it is straightforward to verify that $U^\delta$ is as required. 
	This completes the construction.
	
	Now let $U = \bigcup_{\alpha < \omega_1} U^\alpha$. 
	It is routine to check that $U$ is a normal $\omega_1$-tree, 
	and for all $\alpha < \omega_1$, $U^\alpha$ is a downwards closed subtree of $U$ 
	and $U \res (\alpha+1) = U^\alpha \res (\alpha+1)$. 
	For each $\alpha$, since $U^\alpha$ is downwards closed in $U$, 
	$\h_U \res U^\alpha = \h_{U^\alpha}$. 	
	For each $z \in U$, let $\alpha_z < \omega_1$ be the least ordinal such that 
	$z \in U^{\alpha_z}$. 
	Then either $\alpha_z = 0$ or $\alpha_z$ is a successor ordinal. 
	In the latter case, let $\alpha_z = \beta_z+1$. 
	Note that if $y <_U z$, then $\alpha_y \le \alpha_z$ since $U^{\alpha_z}$ 
	is downwards closed in $U$. 
	For each $z \in U^0 = T$, let $x_z$ denote the root of $T$. 
	For each $z \in U \setminus U^0$, let $x_z$ be the unique member of 
	$U^{\beta_z}_{\beta_z}$ such that $x_z <_U z$ and 
	let $n_z < \omega$ be such that $x_z \in U(\beta_z,x_z,n_z)$.
	
	To show that $U$ is infinitely splitting, consider $z \in U$. 
	Let $\gamma = \h_{U}(z)$. 
	If $\alpha_z = 0$, then at stage $\gamma$ we have that 
	$\h_{U^\gamma}(z) = \gamma$, and we added $\omega$-many immediate successors 
	of $z$ in forming $U^{\gamma+1}$. 
	Suppose that $\alpha_z > 0$. 
	Every element of $U^{\alpha_z} \setminus U^{\beta_z}$ has height at least 
	$\alpha_z$ in $U^{\alpha_z}$. 
	Therefore, $\gamma > \beta_z$. 
	At stage $\gamma$ we have that $\h_{U^\gamma}(z) = \gamma$ 
	and we added $\omega$-many immediate successors of $z$ in forming $U^{\gamma+1}$. 
	This completes the proof that $U$ is infinitely splitting.

	\underline{Claim:} Suppose that $y <_U z$. 
	If $\alpha_y = \alpha_z$ then $x_y = x_z$, and if also $\alpha_y > 0$ 
	then $n_y = n_z$. 
	If $\alpha_y < \alpha_z$ then $x_y <_U x_z <_U z$. 

	\emph{Proof.} First, assume that $\alpha_y = \alpha_z$. 
	If $\alpha_y = 0$ then $x_y = x_z$ is the root of $T$. 
	If $\alpha_y > 0$, 
	then easily $x_y = x_z$ and $n_y = n_z$, for otherwise $y$ and $z$ are incomparable. 
	Secondly, assume that $\alpha_y < \alpha_z$. 
	If $\alpha_y = 0$, then clearly $x_y <_U x_z$ since $x_y$ is the root of $T$. 
	Suppose that $\alpha_y > 0$. 
	Since $y <_U z$, $x_y <_U y$, and $x_z <_U z$, it follows that 
	$x_y$ and $x_z$ are comparable in $U$. 
	But $x_y$ has height $\beta_y$, which is less than the height of $x_z$ 
	which is $\beta_z$. 
	So $x_y <_U x_z$. $\dashv_{\ \text{Claim}}$

	To show that $U$ is special, 
	we define a map $\pi_U : U \to {}^{<\omega} \omega \setminus \{ \emptyset \}$ 
	by induction as follows. 
	For any $z \in U^0 = T$, let $\pi_U(z) = \langle \pi(z) \rangle$. 
	Now consider $z \in U \setminus U^0$ and assume that $\pi_U(y)$ is defined 
	for all $y \in U$ such that $\alpha_y < \alpha_z$. 
	In particular, we have defined $\pi_U(x_z)$. 
	Let $\pi_U(z)$ be the concatenation 
	$\pi_U(x_z)^\frown \langle \pi_{\beta_z,x_z,n_z}(z) \rangle$.

	To show that $\pi_U$ is a specializing map, suppose that $y <_U z$. 
	Assume by induction that for all $c <_U d$ with heights less than 
	$\h_U(z)$, $\pi_U(c) \ne \pi_U(d)$.
	
	\emph{Case 1:} $\alpha_y = \alpha_z$. 
	Then by the claim, $x_y = x_z$, and $n_y = n_z$ in the case that $\alpha_y > 0$. 
	If $\alpha_y = 0$, then $\pi_U(y) = \langle \pi(y) \rangle 
	\ne \langle \pi(z) \rangle = \pi_U(z)$. 
	Suppose that $\alpha_y > 0$. 
	As $\pi_{\beta_y,x_y,n_y}$ is a specializing map on $U(\beta_y,x_y,n_y)$, 
	$\pi_{\beta_y,x_y,n_y}(y) \ne \pi_{\beta_z,x_z,n_z}(z)$. 
	So 
	$$
	\pi_U(y) = \pi_U(x_y)^\frown \langle \pi_{\beta_y,x_y,n_y}(y) \rangle 
	\ne \pi_U(x_z)^\frown \langle \pi_{\beta_z,x_z,n_z}(z) \rangle = \pi_U(z).
	$$

	\emph{Case 2:} $\alpha_y < \alpha_z$. 
	Then by the claim, $x_y <_U x_z <_U z$. 
	By the inductive hypothesis, $\pi_U(x_y) \ne \pi_U(x_z)$. 
	If $\alpha_y = 0$, then $\pi_U(y) = \langle \pi(y) \rangle$, which has length $1$. 
	But $\pi_U(z) = 
	\pi_U(x_z) \con \pi_{\beta_z,x_z,n_z}(z)$, which has length at least $2$ 
	and hence is different form $\pi_U(y)$. 
	If $\alpha_y > 0$, then 
	$$
	\pi_U(y) = \pi_U(x_y)^\frown \langle \pi_{\beta_y,x_y,n_y}(y) \rangle \ne 
	\pi_U(x_z)^\frown \langle \pi_{\beta_z,x_z,n_z}(z) \rangle = \pi_U(z).
	$$
	
	So indeed $U$ is special, and $T$ is an uncountable downwards closed subtree of $U$ 
	which is normal and $2$-splitting. 
	Now for each $x \in U$ with some height $\gamma$, 
	$U_x$ contains $U(\gamma,x,0) \cup \{ x \}$, which is an uncountable 
	downwards closed subtree of $U_x$ which is normal and $2$-splitting.
\end{proof}

A similar construction yields the following proposition.

\begin{proposition}
	Let $2 \le m < n \le \omega$. 
	Assume that $T$ is an $\omega_1$-tree which is $(\ge \! n)$-splitting 
	and for all $x \in T$, $T_x$ contains an uncountable downwards closed 
	$\lem$-splitting subtree. 
	Then $T$ contains an uncountable downwards closed $m$-splitting subtree. 
	Moreover, if $T$ is normal, then there is such a subtree which is normal.
\end{proposition}

\begin{proof}[Proof sketch.]
	Let $W$ be some uncountable downwards closed $\lem$-splitting subtree of $T$. 
	In the case that $T$ is normal, by thinning out further if necessary 
	using Lemma 1.1(a) we may assume that $W$ is normal. 
	Similar to the proof of Proposition 2.2, 
	we build a sequence $\langle U^\alpha : \alpha < \omega_1 \rangle$ 
	of subtrees of $T$ roughly as follows. 
	Let $U^0 = W$, and take unions at limit stages. 
	Assuming that $U^\alpha$ is defined for a fixed $\alpha < \omega_1$, 
	pick for each $x \in U^{\alpha}_\alpha$, 
	$(m-m_x)$-many immediate successors of $x$ in $T \setminus U^\alpha$, where 
	$m_x \le m$ is the number of immediate successors of $x$ in $U^\alpha$. 
	Now construct $U^{\alpha+1}$ by adding 
	above each such immediate successor $y$ 
	some uncountable downwards closed $\lem$-splitting subtree of $T_y$, 
	which we may take to be normal assuming that $T$ is normal. 
	We leave the details to the interested reader.
\end{proof}

\begin{corollary}
	There exists a special Aronszajn tree $T$ which is normal, infinitely splitting, 
	and satisfies that for each $2 \le m < \omega$, 
	$T$ contains an uncountable downwards closed 
	subtree which is normal and $m$-splitting.
\end{corollary}

Let $2 \le m < n < \omega$ and suppose that 
$T$ is a $(\ge \! n)$-splitting $\omega_1$-tree 
which contains a $\lem$-splitting subtree. 
Then $T$ also contains a (non-normal) $n$-splitting subtree, which can be obtained by 
adding new elements with successor height to the $\lem$-splitting subtree to 
ensure $n$-splitting, but at limit levels only including members of the $\lem$-splitting subtree. 
The triviality of this observation displays the importance of the assumption of normality 
on the subtrees.

\begin{proposition}
	Assume $\Diamond$. 
	Then for any $2 \le m < n < \omega$, 
	there exists an infinitely splitting normal Aronszajn tree 
	which contains a normal uncountable downwards closed $m$-splitting subtree, 
	but does not contain a normal uncountable downwards closed $n$-splitting subtree.
\end{proposition}

We leave the proof as an exercise for the interested reader.

We now turn to topological applications of finitely splitting subtrees. 
Let $T$ be an $\omega_1$-tree. 
For any $x \in T$ let $x \up$ denote the set $\{ y \in T : x \le_T y \}$, and for 
any $X \subseteq T$ let $X \up$ denote the union $\bigcup \{ x \up : x \in X \}$.

\begin{definition}[\cite{nyikos}]
	Let $T$ be an $\omega_1$-tree. 
	The \emph{fine wedge topology} on $T$ is the topological space on $T$ with subbase given 
	by sets of the form $x \up$ and $T \setminus (x \up)$ for any $x \in T$.
\end{definition}

It is routine to show that for any $x \in T$, the collection 
of sets of the form $x \up \setminus F \up$, where $F$ is a finite subset 
of $\Imm_T(x)$, is a local base at $x$.

Recall that a topological space $X$ is \emph{Lindel\"{o}f} if every open cover 
of $X$ contains a countable subcover. 
In recent work, Marun characterized the Lindel\"{o}f property of an $\omega_1$-tree 
in the fine wedge 
topology in terms of finitely splitting subtrees.

\begin{thm}[{\cite[Theorem 4.9]{marun}}]
	Let $T$ be an $\omega_1$-tree considered as a 
	topological space with the fine wedge topology. 
	Then $T$ is Lindel\"{o}f iff it does not contain 
	an uncountable downwards closed finitely splitting subtree.
\end{thm}

Going forward we refer to an $\omega_1$-tree which is Lindel\"{o}f in the 
fine wedge topology as a \emph{Lindel\"{o}f tree}. 
Note that if an $\omega_1$-tree $T$ is finitely splitting, then the fine wedge 
topology on $T$ is discrete and hence $T$ is non-Lindel\"{o}f. 
Also, if $T$ has a cofinal branch, then $T$ is non-Lindel\"{o}f 
(\cite[Lemma 4.2]{marun}). 
Less trivial examples of non-Lindel\"{o}f trees are those 
which are Aronszajn and infinitely splitting. 
Marun observed that there exists a normal Aronszajn tree which is infinitely splitting 
and non-Lindel\"{o}f, and that any infinitely splitting Suslin tree is Lindel\"{o}f.

\begin{corollary}
	There exists a normal infinitely splitting non-Lindel\"{o}f special Aronszajn tree.
\end{corollary}

\begin{proof}
	By Corollary 2.4 and Theorem 2.7.
\end{proof}

Marun proved that $\textsf{MA}_{\omega_1}$ implies that every $\omega_1$-tree 
is non-Lindel\"{o}f. 
A variation of his proof shows that $\textsf{MA}_{\omega_1}$ implies that for every 
$2 \le n < \omega$, every $(\ge \! n)$-splitting Aronszajn tree contains an 
uncountable downwards closed $n$-splitting subtree.

In this article, we address the following questions of Marun (\cite{marundiss}, \cite{marun}).

\begin{question}
	Is the topological square of a Lindel\"{o}f tree always Lindel\"{o}f?
\end{question}

\begin{question}
	Is it consistent with \textsf{CH} that every $\omega_1$-tree is non-Lindel\"{o}f?
\end{question}

\begin{question}
	Is it consistent that there exists a Suslin tree, and within the class of infinitely 
	splitting $\omega_1$-trees, being Suslin is equivalent to being Lindel\"{o}f?
\end{question}

Concerning the first question, we show in Section 9 that it is consistent that 
there exists an infinitely splitting $\omega_1$-tree which is Suslin, and 
hence Lindel\"{o}f, but its topological square is not Lindel\"{o}f. 
For the second question, we prove in Section 11 that it is consistent with 
\textsf{CH} that every $\omega_1$-tree is non-Lindel\"{o}f. 
Regarding the third question, we prove in Section 11 that it is consistent that there exist 
a normal infinitely splitting Suslin tree, 
and every infinitely splitting $\omega_1$-tree which does not contain a Suslin subtree 
is non-Lindel\"{o}f. 
By the next proposition, this is the best one can hope for.

\begin{proposition}
	Assume that there exists a Suslin tree. 
	Then there exists a normal Aronszajn tree which is infinitely splitting, 
	Lindel\"{o}f, and not Suslin.
\end{proposition}

\begin{proof}
	It is well-known that if there exists a Suslin tree, then there exists a 
	normal Suslin tree. 
	So by Lemma 1.1(b), if there exists a Suslin tree then there exists a normal 
	infinitely splitting Suslin tree, which we denote by $S$.

	For each limit ordinal $\alpha < \omega_1$ fix a cofinal branch $b_\alpha$ 
	of the tree $S \res \alpha$ which has no upper bound in $S_\alpha$. 
	This is possible since $S$ has countable levels yet $S \res \alpha$ has 
	$2^\omega$-many cofinal branches. 
	Now define a tree $T$ extending $S$ by placing above each $b_\alpha$ a copy of $S$, 
	which we denote by $S^\alpha$. 
	It is routine to check that $T$ is a normal $\omega_1$-tree which is 
	infinitely splitting. 
	If $b$ is a cofinal branch of $T$, then since $S$ has no cofinal branches, there 
	exists some limit ordinal $\alpha < \omega_1$ and some $x \in b \cap S^\alpha$. 
	But $S^\alpha$ is upwards closed in $T$, so $S^\alpha$ contains the uncountable 
	chain $x \up$, contradicting that $S^\alpha$ is Suslin. 
	So $T$ is Aronszajn. 
	For each limit ordinal $\alpha < \omega_1$ 
	let $x_\alpha$ be the root of $S^\alpha$, that is, $x_\alpha$ is the unique upper 
	bound of $b_\alpha$ on level $\alpha$ of $T$. 
	Note that $\{ x_\alpha : \alpha \in \Lim(\omega_1) \}$ 
	is an uncountable antichain of $T$, and hence $T$ is not Suslin.
	
	We claim that $T$ is Lindel\"{o}f. 
	By Theorem 2.7, it suffices to show that $T$ does not contain an uncountable 
	downwards closed finitely splitting subtree. 
	Suppose for a contradiction that $U$ is such a subtree. 
	By finding a subset of $U$ if necessary using Lemma 1.1(a), 
	we may assume without loss of generality that $U$ is normal. 
	\emph{Case 1:} $U \subseteq S$. Then $S$ is not Lindel\"{o}f, which contradicts that 
	$S$ is Suslin. 
	\emph{Case 2:} $U \not \subseteq S$. 
	Then there exists some limit ordinal $\alpha < \omega_1$ such that 
	$U \cap S^\alpha \ne \emptyset$. 
	Since $U$ is normal and $S^\alpha$ is upwards closed in $T$, it follows that 
	$U \cap S^\alpha$ is uncountable. 
	But then $U \cap S^\alpha$ witnesses that $S^\alpha$ is not Lindel\"{o}f, 
	which contradicts that $S^\alpha$ is Suslin.
\end{proof}

On the other hand, if there exists a Suslin tree, then 
there exists a normal infinitely splitting 
Aronszajn tree which contains a Suslin tree but is not Lindel\"{o}f. 
We leave the proof as an exercise for the interested reader.

\section{A Preservation Theorem for Iterated Forcing}

In this section we prove a preservation theorem for iterated forcing related to 
not adding uncountable small-splitting subtrees of a given $\omega_1$-tree. 
One of our main consistency results, proven in Section 11, 
is that for any $2 < n < \omega$, it is consistent 
that every $(\ge \! n)$-splitting normal Aronszajn tree contains an 
uncountable downwards closed normal $n$-splitting 
subtree, but there exists an infinitely splitting normal Aronszajn tree which contains 
no uncountable downwards closed $(< \! n)$-splitting subtree. 
In Section 6 we introduce a forcing for adding such an $n$-splitting subtree, 
and in Section 7 we show that this forcing does not add a $(< \! n)$-splitting subtree to 
any $\omega_1$-tree which did not already contain one. 
But in order to obtain the above consistency result, we need to be able to iterate 
such forcings while preserving these properties. 
To show that this can be done is the purpose of the next theorem.

\begin{thm}
	Let $T$ be an $\omega_1$-tree and let $1 < n \le \omega$. 
	Assume that $T$ does not contain an 
	uncountable downwards closed $(< \! n)$-splitting subtree. 
	Suppose that 
	$\langle \p_i, \dot \q_j : i \le \beta, \ j < \beta \rangle$ 
	is a countable support forcing iteration such that for all $i < \beta$, 
	$\p_i$ forces that $\dot \q_i$ 
	is proper and $\dot \q_i$ forces that $T$ does not contain 
	an uncountable downwards closed 
	$(< \! n)$-splitting subtree. 
	Then $\p_\beta$ is proper and forces that $T$ does not contain 
	an uncountable downwards closed $(< \! n)$-splitting subtree. 	
\end{thm}

The proof of Theorem 3.1 is an elaboration of a proof of Shelah's theorem on the 
preservation of properness under countable support forcing iterations. 
We use the following standard fact from the theory of iterations of proper forcing. 
For a proof of this fact, see \cite[Lemma 31.17]{jechbook}. 

\begin{lemma}
	Let $\langle \p_i, \dot \q_j : i \le \delta, \ j < \delta \rangle$ be a countable 
	support forcing iteration of proper forcings.  
	Let $\lambda$ be a large enough regular cardinal and let $N$ be a countable 
	elementary substructure of $H(\lambda)$ containing the 
	forcing iteration as a member. 
	Assume the following:
	\begin{itemize}
	\item $\alpha \in N \cap \delta$;
	\item $\dot p$ is a $\p_{\alpha}$-name for a condition in $\p_\delta$;
	\item $q$ is an $(N,\p_{\alpha})$-generic condition;
	\item $q \Vdash_\alpha (\dot p \in N \ \land \ 
	\dot p \res \alpha \in \dot G_{\alpha})$.
	\end{itemize}
	Then there exists some $q^+ \in \p_\delta$ such that 
	$q^+ \res \alpha = q$, $q^+$ is $(N,\p_\delta)$-generic, and 
	$q^+ \Vdash_\delta \dot p \in \dot G_\delta$.
\end{lemma}

\begin{proof}[Proof of Theorem 3.1]
	We prove the statement by induction on $\beta$. 
	Fix $T$, $n$, and 
	$\langle \p_i, \dot \q_j : i \le \beta, \ j < \beta \rangle$ 
	as in the theorem. 
	Without loss of generality we may assume that for all $\gamma \le \beta$, 
	$\p_\gamma$ is separative, for otherwise there exists an equivalent 
	forcing iteration which has this property (this step is not necessary, but 
	it eliminates the need for an additional lemma).

	The cases in which $\beta$ is equal to $0$ or is a successor ordinal are immediate. 
	So assume that $\beta$ is a limit ordinal. 
	By the inductive hypothesis, for all $\gamma < \beta$, $\p_\gamma$ is proper and 
	forces that $T$ does not contain an uncountable downwards closed 
	$(< \! n)$-splitting subtree.
	
	Suppose that $p \in \p_\beta$ and $p$ forces that $\dot U$ is a 
	downwards closed $(< \! n)$-splitting subtree of $T$. 
	We find an extension of $p$ which forces that $\dot U$ is countable.

	\underline{Claim:} Suppose that $D$ is a dense open subset of $\p_\beta$, 
	$p^* \le_\beta p$, and $\alpha < \beta$. 
	Then $\p_\alpha$ forces that, if $p^* \res \alpha \in \dot G_\alpha$, then   
	there exists some $\gamma < \omega_1$ so that for all $x \in T_\gamma$, 
	there exists $q_{x} \le_\beta p^*$ in $D$ such that 
	$q_{x} \res \alpha \in \dot G_\alpha$ and 
	$q_{x} \Vdash^V_\beta x \notin \dot U$.

	\emph{Proof.} 
	We prove that any $r \in \p_\alpha$ has an extension in $\p_\alpha$ 
	which forces the conclusion of the claim. 
	So let $r \in \p_\alpha$ be given. 
	Extending $r$ further if necessary, we may assume that $r$ decides whether or 
	not $p^* \res \alpha \in \dot G_\alpha$. 
	If it decides not, then we are done. 
	Otherwise, by separativity $r \le_\alpha p^* \res \alpha$. 
	Let $r^* = r \cup p^* \res [\alpha,\beta)$, 
	which is in $\p_\beta$ and extends $p^*$. 
	Since $D$ is dense in $\p_\beta$, we can fix some $u \le_\beta r^*$ in $D$. 
	Then $u \res \alpha \le_\alpha r$.

	We argue that $u \res \alpha$ is as required. 	
	So let $G_\alpha$ be a generic filter on $\p_\alpha$ which contains $u \res \alpha$. 
	In $V[G_\alpha]$, let $\p_\beta / G_\alpha$ be the suborder of $\p_\alpha$ consisting 
	of those $q \in \p_\beta$ such that $q \res \alpha \in G_\alpha$. 
	So $u \in \p_\beta / G_\alpha$.  
	Recall that in $V$, $\p_\beta$ is forcing equivalent to 
	$\p_\alpha * (\p_\beta / \dot G_\alpha)$. 
	So it makes sense to define in the model $V[G_\alpha]$ 
	$$
	W = \{ x \in T : u \Vdash_{\p_\beta / G_\alpha} x \in \dot U \}.
	$$
	Clearly, $W$ is a downwards closed subtree of $T$ and 
	$u \Vdash_{\p_\beta / G_\alpha} W \subseteq \dot U$. 
	Since $u \le_\beta p$ and $p$ forces that $\dot U$ is $(< \! n)$-splitting, 
	$W$ must be $(< \! n)$-splitting.

	By the inductive hypothesis, $T$ does not contain an uncountable downwards closed 
	$(< \! n)$-splitting subtree in $V[G_\alpha]$. 
	Therefore, $W$ must be countable. 
	Fix some $\gamma < \omega_1$ such that $W \subseteq T \res \gamma$. 
	Now consider any $x \in T_\gamma$. 
	Then $x \notin W$, so by the definition of $W$ 
	there exists some $p_x \le_\beta u$ in $\p_\beta / G_\alpha$ 
	such that $p_x \Vdash_{\p_\beta / G_\alpha} x \notin \dot U$. 
	Since $p_x \in \p_\beta / G_\alpha$, $p_x \res \alpha \in G_\alpha$. 
	Find $v_x \in G_\alpha$ which forces in $\p_\alpha$ over $V$ that 
	$p_x \in \p_\beta / \dot G_\alpha$ and 
	$p_x \Vdash_{\p_\beta / \dot G_\alpha}^{V[\dot G_\alpha]} x \notin \dot U$. 
	By separativity, $v_x \le_\alpha p_x \res \alpha$. 
	Now let $q_x = v_x \cup p_x \res [\alpha,\beta)$. 
	Then $q_x \le_\beta p_x \le_\beta u$. 
	It is easy to check that $q_x \Vdash^V_\beta x \notin \dot U$. 
	As $u \in D$ and $D$ is open, $q_x \in D$. $\dashv_{\ \text{Claim}}$

	We now work towards finding a condition $q \le_\beta p$ which forces that 
	$\dot U$ is countable. 
	Let $\lambda$ be a large enough regular cardinal and let $N \prec H(\lambda)$ 
	be countable such that $N$ contains as members the objects 
	$\langle \p_i, \dot \q_j : i \le \beta, \ j < \beta \rangle$, 
	$p$, $T$, and $\dot U$. 
	Fix the following objects:
	\begin{itemize}
	\item an enumeration $\langle D_m : m < \omega \rangle$ of all dense open 
	subsets of $\p_\beta$ which lie in $N$;
	\item an enumeration $\langle x_m : m < \omega \rangle$ of $T_{N \cap \omega_1}$;
	\item an increasing sequence $\langle \alpha_m : m < \omega \rangle$ of ordinals 
	in $N \cap \beta$ which is cofinal in $\sup(N \cap \beta)$, where $\alpha_0 = 0$.
	\end{itemize}

	We define by induction sequences $\langle \dot p_m : m < \omega \rangle$ and 
	$\langle q_m : m < \omega \rangle$ satisfying that for all $m < \omega$:
	\begin{enumerate}
		\item $\dot p_m$ is a $\p_{\alpha_m}$-name for a condition 
		in $\p_\beta$, and $\dot p_0$ is a $\p_0$-name for $p$;
		\item $q_m$ is an $(N,\p_{\alpha_m})$-generic condition  
		and $q_m \Vdash_{\alpha_m} (\dot p_m \in N \ \land \ 
		\dot p_m \res \alpha_m \in \dot G_{\alpha_m})$;
		\item $q_{m+1} \res \alpha_m = q_m$, and 
		$q_{m+1}$ forces in $\p_{\alpha_{m+1}}$ that:
		\begin{enumerate}
			\item $\dot p_{m+1} \le_\beta \dot p_m$;
			\item $\dot p_{m+1} \in D_m$;
			\item $\dot p_{m+1} \Vdash^V_\beta x_m \notin \dot U$.
		\end{enumerate}
	\end{enumerate}
	Note that the statements above imply that for all $m < \omega$, 
	$q_m \Vdash_{\alpha_m} \dot p_m \le_\beta p$.

	Begin by letting $q_0$ be the empty condition in $\p_0$ and letting 
	$\dot p_0$ be a $\p_0$-name for $p$. 
	Now let $m < \omega$ and assume that $\dot p_m$ and $q_m$ are defined as described. 
	Note that properties (1) and (2) imply that the assumptions of Lemma 3.2 hold, where 
	we let $\delta = \alpha_{m+1}$, 
	$\alpha = \alpha_m$, $\dot p = \dot p_m \res \alpha_{m+1}$, and $q = q_m$. 
	So fix $q_{m+1} \in \p_{\alpha_{m+1}}$ such that 
	$q_{m+1} \res \alpha_m = q_{m}$, 
	$q_{m+1}$ is $(N,\p_{\alpha_{m+1}})$-generic, 
	and $q_{m+1} \Vdash_{\alpha_{m+1}} 
	\dot p_m \res \alpha_{m+1} \in \dot G_{\alpha_{m+1}}$.

	To define the $\p_{\alpha_{m+1}}$-name $\dot p_{m+1}$, 
	consider a generic filter 
	$G_{\alpha_{m+1}}$ for $\p_{\alpha_{m+1}}$ which contains $q_{m+1}$. 
	Since $q_{m+1}$ is $(N,\p_{\alpha_{m+1}})$-generic, 
	$N[G_{\alpha_{m+1}}] \cap V = N$. 
	Let $G_{\alpha_m} = G_{\alpha_{m+1}} \res \alpha_m$. 
	Since $q_{m+1} \res \alpha_{m} = q_m$, $q_m \in G_{\alpha_m}$. 
	Let $p_m = \dot p_m^{G_{\alpha_m}}$. 
	Then $p_m \in N$ by (2). 
	By the choice of $q_{m+1}$, 
	$p_m \res \alpha_{m+1} \in G_{\alpha_{m+1}}$. 
	Also, $p_m \le_\beta p$. 
	Applying the claim for $D = D_m$, $p^* = p_m$, and $\alpha = \alpha_{m+1}$, 
	in $V[G_{\alpha_{m+1}}]$ 
	there exists some $\gamma < \omega_1$ so that 
	for all $x \in T_\gamma$, there exists $q_x \le_\beta p_m$ 
	in $D_m$ such that $q_x \res \alpha_{m+1} \in G_{\alpha_{m+1}}$ 
	and $q_x \Vdash^V_\beta x \notin \dot U$. 
	Since the relevant parameters are in $N[G_{\alpha_{m+1}}]$, 
	we can assume that $\gamma \in N[G_{\alpha_{m+1}}] \cap \omega_1 = N \cap \omega_1$.
	
	Letting $x = x_m \res \gamma$, we can find a condition 
	$p_{m+1} \le_\beta p_m$ in $D_m$ such that 
	$p_{m+1} \res \alpha_{m+1} \in G_{\alpha_{m+1}}$ 
	and $p_{m+1} \Vdash^V_\beta x_m \res \gamma \notin \dot U$. 
	By elementarity, we may assume that 
	$p_{m+1} \in N[G_{\alpha_{m+1}}]$. 
	But $N[G_{\alpha_{m+1}}] \cap V = N$, so $p_{m+1} \in N$. 
	In summary, $p_{m+1} \le_\beta p_m$, $p_{m+1} \in N \cap D_m$, 
	$p_{m+1} \res \alpha_{m+1} \in G_{\alpha_{m+1}}$, and 
	$p_{m+1} \Vdash^V_\beta x_m \res \gamma \notin \dot U$.
	Since $\dot U$ is forced to be downwards closed, 
	$p_{m+1} \Vdash^V_\beta x_m \notin \dot U$. 
	Now let $\dot p_{m+1}$ be a $\p_{\alpha_{m+1}}$-name for a member of $\p_\beta$ 
	which is forced to 
	satisfy the above properties in the case that $q_{m+1} \in \dot G_{\alpha_{m+1}}$.

	This completes the construction. 
	Let $q = \bigcup_m q_m$, which is in $\p_{\beta}$. 
	We argue that $q \le_\beta p$, $q$ is $(N,\p_\beta)$-generic, 
	and $q$ forces that $\dot U$ is countable. 
	Fix a generic filter $G_\beta$ for $\p_\beta$ which contains $q$ 
	and let $U = \dot U^{G_\beta}$. 
	For each $m < \omega$ let 
	$G_{\alpha_m} = G_\beta \res \alpha_m$ and 
	$p_m = \dot p_m^{G_{\alpha_m}}$. 
	
	We claim that for all $m < \omega$, $q \le_\beta p_m$, and 
	in particular, $q \le_\beta p_0 = p$. 
	Since $q \res \alpha_m = q_m$, 
	$q_m \Vdash_{\alpha_m} \dot p_m \res \alpha_m \in \dot G_{\alpha_m}$, 
	and $\p_{\alpha_m}$ is separative, it follows that 
	$q \res \alpha_m \le_{\alpha_m} p_m \res \alpha_m$. 
	Hence, for all $m < k < \omega$, since $p_k \le_\beta p_m$ we have that 
	$q \res \alpha_k \le_{\alpha_k} p_k \res \alpha_k \le_{\alpha_k} p_m \res \alpha_k$. 
	Now $\dom(q) \subseteq \bigcup_k \alpha_k = \sup(N \cap \beta)$, 
	and since $p_m \in N$, also $\dom(p_m) \subseteq \sup(N \cap \beta)$. 
	So in fact for all $\delta < \beta$, $q \res \delta \le_\delta p_m \res \delta$, 
	which implies that $q \le_\beta p_m$. 
	
	Since $G_\beta$ was an arbitrary generic filter which contains $q$, 
	$q$ forces in $\p_\beta$ 
	that for all $m < \omega$, $q \le_\beta \dot p_m$, and hence that 
	$\dot p_m \in \dot G_\beta \cap N$. 
	So by (3), $q$ forces in $\p_\beta$ that for all $m < \omega$, 
	$\dot p_{m+1} \in D_m \cap N \cap \dot G_\beta$. 
	This proves that $q$ is $(N,\p_\beta)$-generic. 
	Also by (3), $q$ forces in $\p_\beta$ that for all $m < \omega$, 
	$p_{m+1}$ is a condition in $\dot G_\beta$ 
	which forces over $V$ that $x_m \notin \dot U$, 
	and therefore in fact $x_m \notin \dot U$. 
	Since $\dot U$ is forced to be downwards closed, 
	$q$ forces that $\dot U \subseteq T \res (N \cap \omega_1)$ 
	and hence $\dot U$ is countable.
\end{proof}

\section{Promises and Generalized Promises}

The forcing poset for adding an $n$-splitting subtree of a given Aronszajn tree, 
which is introduced in Section 6, consists of conditions with two components. 
The first component of a condition 
is a countable approximation of the generic subtree, and the second 
component is a countable collections of 
objects which we call \emph{generalized promises} which constrains 
the growth of the first component to ensure that the forcing is well-behaved.  
In this section we describe promises and generalized promises and adapt 
a result about promises due to Jensen to the context of this article.

\begin{definition}[Promises]
	Let $T$ be an $\omega_1$-tree. 
	For any positive $m < \omega$ and $\gamma < \omega_1$, 
	a \emph{promise on $T$ with dimension $m$ and base level $\gamma$} 
	is an uncountable downwards closed subtree $U$ 
	of some derived tree $T_{\vec x}$ of $T$ with dimension $m$ and each 
	member of $\vec x$ having height $\gamma$ 
	such that every element of $U$ has uncountably many 
	elements above it in $U$.
\end{definition}

Note that if $T$ is normal, then so is any promise on $T$.

\begin{lemma}
	Let $T$ be an $\omega_1$-tree.
	Suppose that $W$ is an uncountable downwards closed subtree of some 
	derived tree $T_{\vec x}$ of $T$. 
	Then $W$ contains a promise with root $\vec x$.
\end{lemma}

\begin{proof}
	By Lemma 1.1(a).
\end{proof}

We need the following result of Jensen (originally stated without the use of the 
word ``promise'').

\begin{proposition}[{\cite[Chapter 6, Lemma 7]{devlinj}}]
	Let $T$ be an Aronszajn tree and let $U$ be a promise 
	on $T$ with dimension $m$ and base level $\gamma$. 
	Then for any countable ordinal $\alpha \ge \gamma$ and for all 
	$\vec a \in U \cap T_\alpha^m$, there exists a countable ordinal $\beta > \alpha$ 
	and a sequence 
	$\langle \vec b_k : k < \omega \rangle$ consisting of pairwise disjoint 
	tuples in $U \cap T_\beta^m$ above $\vec a$.
\end{proposition}

\begin{definition}
	Let $T$ be an $\omega_1$-tree, let 
	$\alpha < \beta < \omega_1$, and let $X \subseteq T_\beta$ be finite. 
	We say that $X$ has \emph{unique drop-downs to $\alpha$} 
	if the map $x \mapsto x \res \alpha$ is injective on $X$. 
	More generally, for any $1 \le n < \omega$ 
	we say that $X$ has \emph{$\len$-to-one drop-downs to $\alpha$} if 
	for all $y \in T_\alpha$, the set $\{ x \in X : x \res \alpha = y \}$ has 
	size at most $n$.
\end{definition}

Note that $(\le \! 1)$-to-one drop-downs is equivalent to unique drop-downs.

The following lemma is easy.

\begin{lemma}
	Let $T$ be an $\omega_1$-tree, let $n < \omega$ be positive, 
	let $\alpha < \gamma < \beta < \omega_1$, 
	and let $X \subseteq T_\beta$ be finite. 
	If $X$ has $\len$-to-one drop-downs to $\alpha$, then 
	$X$ has $\len$-to-one drop-downs to $\gamma$ and 
	$X \res \gamma$ has $\len$-to-one drop-downs to $\alpha$. 
	If $X$ has unique drop-downs to $\gamma$ and $X \res \gamma$ has 
	$\len$-to-one drop-downs to $\alpha$, then $X$ has $\len$-to-one dropdowns to $\alpha$.
\end{lemma}

The next lemma adapts Proposition 4.3 to the context of 
this article and is essential to adding generalized 
promises to conditions in the forcing introduced in Section 6.

\begin{lemma}
	Let $T$ be an Aronszajn tree and suppose that $U$ is a promise on $T$ 
	with dimension $m$ and base level $\beta$. 
	Then:
	\begin{itemize}
	\item there exists an increasing sequence $\langle \delta_n : n < \omega \rangle$ 
	of countable ordinals with supremum $\delta$, where $\delta_0 > \beta$;
	\item there exists a sequence 
	$\langle \vec b_n : n < \omega \rangle$ of pairwise disjoint 
	tuples in $U \cap T_\delta^m$, where we let 
	$X_n$ be the elements of $\vec b_n$ for all $n < \omega$;
	\item for all $0 < k < \omega$, 
	$X_0 \cup \cdots \cup X_{k}$ has unique drop-downs to $\delta_{k-1}$;
	\item $X_0 \cup X_1$ has $(\le \! 2)$-to-one drop-downs to $\beta$ and 
	for all $1 < k < \omega$, $X_0 \cup \cdots \cup X_k$ has 
	$(\le \! 2)$-to-one drop-downs to $\delta_{k-2}$;
	\item for all positive $k < \omega$, 
	$(\bigcup_n X_n) \res \delta_{k-1} = 
	(X_0 \cup \cdots \cup X_k) \res \delta_{k-1}$.
	\end{itemize}
\end{lemma}

\begin{proof}
	Let $\vec x$ be the root of $U$. 
	We define a sequence 
	$\langle \delta_n : n < \omega \rangle$ and tuples 
	$\vec b_n^0$ and $\vec b_n^1$ in $U \cap T_{\delta_n}^m$ for all $n < \omega$ 
	by induction. 
	We maintain inductively that for all $k < n < \omega$ and $j < 2$, 
	$\vec b_n^j \res \delta_k = \vec b_k^0$. 

	For the base case, apply Proposition 4.3 to find some countable ordinal $\delta_0 > \beta$ 
	such that there exist infinitely many pairwise disjoint members of 
	$U \cap T_{\delta_0}^m$ above $\vec x$. 
	Pick two among these disjoint tuples and call them $\vec b_0^0$ and $\vec b_0^1$. 	
	Now let $n < \omega$ and assume that 
	$\delta_n$, $\vec b_n^0$, and $\vec b_n^1$ are defined. 
	Apply Proposition 4.3 to find some countable ordinal $\delta_{n+1} > \delta_n$ such that 
	there exist infinitely many pairwise disjoint members of $U \cap T_{\delta_{n+1}}^m$ 
	above $\vec b_n^0$. 
	Take two such disjoint tuples and call them $\vec b_{n+1}^0$ and $\vec b_{n+1}^1$. 
	Note that the inductive hypothesis is maintained. 
	This completes the induction. 
	Let $\delta = \sup_n \delta_n$.
	
	For every $n < \omega$ pick some tuple $\vec b_n$ in 
	$U \cap T_{\delta}^m$ above $\vec b_n^1$. 
	Consider $k < n < \omega$. 
	Then $\vec b_k \res \delta_k = \vec b_k^1$, which is disjoint from $\vec b_k^0$. 
	On the other hand, by our inductive hypothesis 
	$$
	\vec b_n \res \delta_k = (\vec b_n \res \delta_n) \res \delta_k = 
	\vec b_n^1 \res \delta_k = \vec b_k^0.
	$$
	So $\vec b_k \res \delta_k$ and $\vec b_n \res \delta_k$ are disjoint. 
	It follows that $\vec b_k$ and $\vec b_n$ are disjoint as well. 

	Now fix $0 < k < \omega$. 
	Consider $n_0 < n_1 \le k$. 
	Then by the last paragraph, 
	$\vec b_{n_0} \res \delta_{n_0}$ and $\vec b_{n_1} \res \delta_{n_0}$ are disjoint. 
	Since $\delta_{k-1} \ge \delta_{n_0}$, 
	$\vec b_{n_0} \res \delta_{k-1}$ and 
	$\vec b_{n_1} \res \delta_{k-1}$ are also disjoint. 
	For any $n \le k$, $\vec b_n \res \beta = \vec x$, which is injective, 
	so $X_n$ has unique drop-downs to $\beta$ and hence also to $\delta_{k-1}$. 
	It easily follows from these observations that 
	$X_0 \cup \cdots \cup X_{k}$ has unique drop-downs to $\delta_{k-1}$.
	
	By the above, $X_0 \cup X_1$ has unique drop-drowns to $\delta_0$, 
	and since $(X_0 \cup X_1) \res \delta_0$ consists of the elements of 
	$\vec b_0^0$ and $\vec b_0^1$, which are disjoint and above $\vec x$, 
	$(X_0 \cup X_1) \res \delta_0$ has $(\le \! 2)$-to-one drop-downs to $\beta$. 
	By Lemma 4.5, $X_0 \cup X_1$ has $(\le \! 2)$-to-one drop-downs to $\beta$. 
	Now let $1 < k < \omega$ and we show that $X_0 \cup \cdots \cup X_k$ 
	has $(\le \! 2)$-to-one drop-downs to $\delta_{k-2}$. 
	As proved above, 
	$X_0 \cup \cdots \cup X_{k-1}$ has unique drop-downs to $\delta_{k-2}$. 
	Now $\vec b_k \res \delta_k = \vec b_k^1$ and 
	$\vec b_{k-1} \res \delta_{k-1} = \vec b_{k-1}^1$. 
	By our inductive hypothesis, it follows that 
	$\vec b_k \res \delta_{k-2} = (\vec b_k \res \delta_k) \res \delta_{k-2} = 
	\vec b_{k}^1 \res \delta_{k-2} = \vec b_{k-2}^0$ and 
	$\vec b_{k-1} \res \delta_{k-2} = (\vec b_{k-1} \res \delta_{k-1}) \res \delta_{k-2} = 
	\vec b_{k-1}^1 \res \delta_{k-2} = \vec b_{k-2}^0$. 
	It easily follows that 
	$X_{k-1} \cup X_k$ has $(\le \! 2)$-to-one drop-downs to $\delta_{k-2}$ and 
	so $X_0 \cup \cdots \cup X_{k}$ has $(\le \! 2)$-to-one drop-downs to $\delta_{k-2}$.
	
	Finally, consider $0 < k < \omega$.
	For any $k < n < \omega$, we have that 
	$\vec b_n \res \delta_{k-1} = \vec b_{k-1}^0 = \vec b_k \res \delta_{k-1}$. 
	It follows that 
	$(\bigcup_n X_n) \res \delta_{k-1} = (X_0 \cup \ldots \cup X_k) \res \delta_{k-1}$.	
\end{proof}

For many of the purposes of this article, it would suffice to define our forcing poset 
using promises. 
But for the preservation of Suslin trees, we need a more general notion 
(also see the remark after \cite[Definition 4.2]{AS2}).

\begin{definition}[Generalized Promises]
	Let $T$ be an $\omega_1$-tree. 
	Let $m < \omega$ be positive and $\gamma < \omega_1$. 
	A \emph{generalized promise on $T$ with dimension $m$ and base level $\gamma$} is a 
	function $\mathcal U$ with domain $[\gamma,\omega_1)$ satisfying that, for some 
	derived tree $T_{\vec x}$ of $T$, for all $\gamma \le \beta < \omega_1$:
	\begin{enumerate}
	\item $\mathcal U(\beta)$ is a non-empty countable collection of infinite
	subsets of $T_{\vec x} \cap T_\beta^m$;
	\item for every $\mathcal Y \in \mathcal U(\beta)$, 
	$\mathcal Y$ is 
	\emph{well-distributed} in the sense that 
	for every finite set $t \subseteq T_\beta$, there exists some element of $\mathcal Y$ 
	which is disjoint from $t$;
	\item for all $\gamma \le \alpha < \beta < \omega_1$ 
	and for any $\mathcal Y \in \mathcal U(\beta)$, 
	$\mathcal Y \res \alpha \in \mathcal U(\alpha)$.
	\end{enumerate}
\end{definition}

It is implicit in the above that $\vec x$ has size $m$ and its elements have 
height at most $\gamma$. 
Note that in (2), it suffices that this statement holds for 
$\mathcal U(\gamma)$. 
For assume this and 
suppose that $\gamma < \beta < \omega_1$ and $\mathcal Y \in \mathcal U(\beta)$. 
Then by (3), $\mathcal Y \res \gamma \in \mathcal U(\gamma)$. 
Let $t \subseteq T_\beta$ be finite. 
Then $t \res \gamma \subseteq T_\gamma$ is finite, so we can find some 
$\vec b_0 \in \mathcal Y \res \gamma$ which is disjoint from $t \res \gamma$. 
Now choose $\vec b \in \mathcal Y$ such that $\vec b \res \gamma = \vec b_0$. 
Then $\vec b$ is disjoint from $t$.

A generalized promise $\mathcal U$ 
is essentially an $\omega_1$-tree whose elements are sets $\mathcal Y$ 
such that for some $\gamma \le \beta < \omega_1$, $\mathcal Y \in [T_\beta^m]^\omega$ 
is well-distributed, and if $\mathcal Y$ and $\mathcal Z$ are members of the tree 
in $[T_\beta^m]^\omega$ and $[T_\delta^m]^\omega$ respectively, where $\beta < \delta$, 
then $\mathcal Y$ is below $\mathcal Z$ in the tree 
iff $\mathcal Y = \mathcal Z \res \beta$. 
For all $\alpha < \omega_1$, level $\alpha$ of this tree is equal to 
$\mathcal U(\gamma+\alpha)$. 
The union of any cofinal branch of this tree is the tail of some promise on $T$.

\begin{lemma}
	Let $T$ be an $\omega_1$-tree. 
	Let $m < \omega$ be positive and $\zeta < \omega_1$. 
	Suppose that $U$ is a promise with dimension $m$ and base level $\zeta$. 
	Then for some countable ordinal $\gamma < \omega_1$ greater than $\zeta$, 
	the function $\mathcal U$ with domain $[\gamma,\omega_1)$ 
	defined by $\mathcal U(\beta) = \{ U \cap T_\beta^m \}$ is a generalized 
	promise with dimension $m$ and base level $\gamma$.
\end{lemma}

\begin{proof}
	By Proposition 4.3, we can fix some $\gamma < \omega_1$ greater than $\zeta$ 
	for which there exists a pairwise 
	disjoint sequence $\langle \vec b_n : n < \omega \rangle$ of tuples in 
	$U \cap T_\gamma^m$ above the root of $U$. 
	Then for all finite $t \subseteq T_\gamma$, there is some $n$ such that 
	$\vec b_n$ is disjoint from $t$. 
	For all $\gamma \le \alpha < \beta < \omega_1$, 
	the equality $(U \cap T_\beta^m) \res \alpha = U \cap T_\alpha^m$ 
	follows from the fact that $U$ is downwards closed and every element of $U$ 
	has uncountably many elements above it.
\end{proof}

\section{Splitting Subtree Functions}

For Sections 5-8 and 10, we fix $T$ and $\mathcal S$ satisfying that $T$ is a normal 
Aronszajn tree, $\mathcal S \subseteq \omega \setminus 2$ is non-empty, 
and for any $x \in T$, 
$|\Imm_T(x)| \ge \sup(\mathcal S)$. 
Let $s = \min(\mathcal S)$. 
In the next section we introduce a forcing which adds 
an uncountable downwards closed normal subtree $U$ of $T$ such that for all $x \in U$, 
$| \Imm_U(x) | \in \mathcal S$. 
We are mainly interested in two cases: when $\mathcal S = \omega \setminus 2$, 
so $T$ is infinitely splitting and we add a finitely splitting subtree of $T$, 
and when $\mathcal S = \{ s \}$ and we add an $s$-splitting subtree of $T$.

As mentioned previously, conditions in the forcing poset we introduce in Section 6 
have two components, the first being a countable approximation of the generic subtree 
and the second being a countable collection of generalized promises. 
We now develop some basic information about the countable approximations 
and how they relate to generalized promises.

\begin{definition}[Subtree Functions]
	A \emph{subtree function with top level $\beta < \omega_1$} is a function 
	$F : T \res (\beta+1) \to 2$ satisfying:
	\begin{enumerate}
	\item the value of $F$ on the root of $T$ equals $1$;
	\item if $F(a) = 1$, then for all $\gamma < \h_T(a)$, 
	$F(a \res \gamma) = 1$;
	\item for all $\alpha < \beta$ and $x \in T_\alpha$, 
	if $F(x) = 1$ then there exists some $y \in T_\beta$ above $x$ 
	such that $F(y) = 1$.
	\end{enumerate}
\end{definition}

If $F$ is a subtree function with top level $\beta$ and $U$ is a promise with 
dimension $m$ and base 
level less than $\beta$, then we say that $F$ \emph{fulfills} $U$ if 
for every finite set $t \subseteq T_\beta$ there exists some 
$(x_0,\ldots,x_{m-1}) \in U \cap T_\beta^m$ disjoint from $t$ 
such that $F(x_i) = 1$ for all $i < m$. 
This definition of fulfillment has as a precursor Shelah's notion of fulfillment 
for his specializing function of \cite[Chapter V, Definition 6.4]{properimproper}, and 
is similar to that used by Lamei Ramandi and Moore 
in their subtree forcing of \cite[Remark 5.2]{mooreramandi}. 
The following definition broadens the notion of fulfillment in a natural way to 
generalized promises, and is similar to that of Abraham-Shelah \cite[Definition 4.1]{AS2}.

\begin{definition}[Fulfilling Generalized Promises]
	Let $F$ be a subtree function with top level $\beta < \omega_1$ 
	and let $\mathcal U$ be a generalized promise with dimension $m$ 
	and base level less than or equal to $\beta$. 
	We say that $F$ \emph{fulfills} $\mathcal U$ if for any finite set 
	$t \subseteq T_\beta$ and for any $\mathcal Y \in \mathcal U(\beta)$, 
	there exists some $(a_0,\ldots,a_{m-1}) \in \mathcal Y$ 
	whose elements are not in $t$ such that $F(a_i) = 1$ for all $i < m$.
\end{definition}

\begin{definition}[$\mathcal S$-Splitting Subtree Functions]
	An \emph{$\mathcal S$-splitting subtree function with top level $\beta < \omega_1$} 
	is any subtree function $F$ with top level $\beta$ satisfying that for all 
	$x \in T \res \beta$, 
	if $F(x) = 1$, then  
	$$
	| \{ y \in \Imm_T(x) : F(y) = 1 \} | \in \mathcal S.
	$$
\end{definition}

Two basic results we need concerning the interaction of 
subtree functions and generalized promises 
are \emph{extension} and \emph{adding generalized promises}, which we prove next. 
The first result say that if an $\mathcal S$-splitting subtree function fulfills countably 
many generalized promises, then we can extend the subtree function 
arbitrarily high while still fulfilling the same generalized promises. 
The second result says that we can extend an $\mathcal S$-splitting subtree function 
to satisfy an additional generalized promise which is suitable for it in some sense.

\begin{proposition}[Extension]
	Suppose that $F$ is an $\mathcal S$-splitting subtree function with top level $\beta$ 
	and $\Gamma$ is a countable collection of generalized promises 
	whose base levels are less than or equal to $\beta$. 
	Assume that $F$ fulfills every generalized promise in $\Gamma$. 
	Let $\beta \le \gamma < \omega_1$, let $X \subseteq T_\gamma$ be finite with 
	$\les$-to-one drop-downs to $\beta$, and assume that for all $x \in X$, 
	$F(x \res \beta) = 1$. 
	Then there exists an $\mathcal S$-splitting subtree function $F^+$ extending $F$ 
	with top level $\gamma$ such that $F^+(x) = 1$ for all $x \in X$ 
	and $F^+$ fulfills every generalized promise in $\Gamma$.
\end{proposition}

\begin{proof}
	The proof is by induction on $\gamma$, where the base case $\gamma = \beta$ is immediate.

	\emph{Successor case:} Assume that the statement is true for some countable 
	ordinal $\gamma \ge \beta$, and we prove that it is true for $\gamma+1$. 
	So let $X \subseteq T_{\gamma+1}$ be finite with $\les$-to-one 
	drop-downs to $\beta$ such that for all $x \in X$, $F(x \res \beta) = 1$. 
	By Lemma 4.5, $X$ has $\les$-to-one drop-downs to $\gamma$ and 
	$X \res \gamma$ has $\les$-to-one drop-downs to $\beta$. 
	Obviously, for all $y \in X \res \gamma$, $F(y \res \beta) = 1$. 
	By the inductive hypothesis, we can fix an $\mathcal S$-splitting subtree function $G$ 
	with top level $\gamma$ such that $G(x \res \gamma) = 1$ for all $x \in X$ 
	and $G$ fulfills every generalized promise in $\Gamma$.

	Next we extend $G$ to $G^+$ with top level $\gamma+1$. 
	To help with the construction of $G^+$, fix the following objects:
	\begin{itemize}
		\item an enumeration $\langle y_n : n < \omega \rangle$ of $T_\gamma$;
		\item an enumeration $\langle \mathcal U_n : n < \omega \rangle$ of $\Gamma$;
		\item for each $n < \omega$, an enumeration 
		$\langle \mathcal Y^n_k : k < \omega \rangle$ of $\mathcal U_n(\gamma+1)$ 
		(possibly with repetitions);
		\item an enumeration $\langle t_n : n < \omega \rangle$ of all finite 
		subsets of $T_{\gamma+1}$;
		\item a bijection $h : \omega \to 2 \times \omega^3$.
	\end{itemize}
	We define by induction an increasing sequence $\langle Z_n : n < \omega \rangle$ 
	of finite subsets of $T_{\gamma+1}$. 
	Our inductive hypothesis is that for all $n < \omega$ and $y \in T_{\gamma}$, 
	$y$ has at most $s$-many immediate successors in $Z_n$.

	To begin, let $Z_0 = X$. 
	Since $X$ has $\les$-to-one drop-downs to $\gamma$, 
	the inductive hypothesis holds. 	
	Now let $n < \omega$ and assume that $Z_n$ is defined and satisfies the 
	inductive hypothesis. 
	Let $h(n) = (i,k_0,k_1,k_2)$.

	\emph{Case 1:} $i = 0$. 
	Consider $y_{k_0}$. 
	If $G(y_{k_0}) = 0$, then let $Z_{n+1} = Z_n$. 
	Assume that $G(y_{k_0}) = 1$. 
	Since $y_{k_0}$ has at least $s$-many immediate successors in $T$ and 
	$y_{k_0}$ has at most $s$-many immediate 
	successors in $Z_n$, we can find a finite set 
	$Z_{n+1}$ containing $Z_n$ such that $y_{k_0}$ has exactly $s$-many immediate 
	successors in $Z_{n+1}$ and every member of $Z_{n+1} \setminus Z_n$ is an 
	immediate successor of $y_{k_0}$. 
	Clearly, the inductive hypothesis holds.

	\emph{Case 2:} $i = 1$. 
	Consider $\mathcal U_{k_0}$, 
	$\mathcal Y^{k_0}_{k_1}$, and $t_{k_2}$. 
	Let $l$ be the dimension of $\mathcal U_{k_0}$. 
	Applying the fact that $G$ fulfills $\mathcal U_{k_0}$, we can find 
	some $\vec a = (a_0,\ldots,a_{l-1}) \in \mathcal Y^{k_0}_{k_1} \res \gamma$ 
	which is disjoint from $(Z_n \cup t_{k_2}) \res \gamma$ and satisfies that 
	for all $i < l$, $G(a_i) = 1$. 
	Fix $\vec a^+ = (a_0^+,\ldots,a_{l-1}^+) \in \mathcal Y^{k_0}_{k_1}$ above $\vec a$. 
	Since $\vec a$ is injective, $\vec a^+$ has unique drop-downs to $\gamma$. 
	Note that $\vec a^+$ is disjoint from $Z_n \cup t_{k_2}$. 
	Define $Z_{n+1} = Z_n \cup \{ a_0^+,\ldots,a_{l-1}^+ \}$. 
	By the inductive hypothesis for $n$ and the fact that 
	$\{ a_0,\ldots,a_{l-1} \}$ is disjoint from $Z_n \res \gamma$, 
	the inductive hypothesis is maintained.

	This completes the construction. 
	Now define $G^+$ extending $G$ so that for all $x \in T_{\gamma+1}$, 
	$G^+(x) = 1$ if $x \in \bigcup_n Z_n$, and $G^+(x) = 0$ otherwise. 
	It is routine to check that $G^+$ is as required.

	\emph{Limit case:} Assume that $\delta > \beta$ is a countable limit ordinal and for 
	all $\gamma$ with $\beta \le \gamma < \delta$, the result holds for $\gamma$. 
	Let $X \subseteq T_\delta$ be a finite set which has $\les$-to-one 
	drop-downs to $\beta$ such that for all $x \in X$, $F(x \res \beta) = 1$. 
	To help with the construction, fix the following objects:
	\begin{itemize}
		\item an increasing sequence of ordinals $\langle \delta_n : n < \omega \rangle$ 
		cofinal in $\delta$ with $\delta_0 = \beta$;
		\item an enumeration $\langle y_n : n < \omega \rangle$ of $T \res \delta$;
		\item an enumeration $\langle \mathcal U_n : n < \omega \rangle$ of $\Gamma$;
		\item for each $n < \omega$, an enumeration 
		$\langle \mathcal Y^n_k : k < \omega \rangle$ of $\mathcal U_n(\delta)$ 
		(possibly with repetitions);
		\item an enumeration $\langle t_n : n < \omega \rangle$ of all finite 
		subsets of $T_{\delta}$;
		\item a surjection $h : \omega \to 2 \times \omega^3$ such that every 
		element of the codomain has an infinite preimage.
	\end{itemize}

	We define by induction 
	an increasing sequence $\langle Z_n : n < \omega \rangle$ 
	of finite subsets of $T_{\delta}$ together with a sequence 
	$\langle F_n : n < \omega \rangle$ satisfying that for all $n < \omega$:
	\begin{enumerate}
		\item $Z_n$ has $\les$-to-one drop-downs to $\delta_n$;
		\item $F_n$ is an $\mathcal S$-splitting subtree function 
		with top level $\delta_n$, $F_n(z \res \delta_n) = 1$ 
		for all $z \in Z_n$, and $F_n \subseteq F_{n+1}$.
	\end{enumerate}

	To begin, let $F_0 = F$ and $Z_0 = X$. 
	Now let $n < \omega$ and assume that $Z_n$ and $F_n$ are defined 
	and satisfy (1) and (2). 
	By Lemma 4.5, $Z_n \res \delta_{m+1}$ has $\les$-to-one drop-downs to $\delta_m$. 
	Let $h(n) = (i,k_0,k_1,k_2)$. 
	As a first step, 
	we define $Z_{n+1}$ which has $\les$-to-one drop-downs to $\delta_{n}$ 
	such that for all $x \in Z_{n+1}$, $F_n(x \res \delta_n) = 1$.

	\emph{Case 1:} $i = 0$. 
	Consider $y_{k_0}$. 
	If either (a) $\h_T(y_{k_0}) > \delta_n$, (b) $\h_T(y_{k_0}) \le \delta_n$ 
	and $F_n(y_{k_0}) = 0$, or (c) $\h_T(y_{k_0}) \le \delta_n$ and 
	$y_{k_0} \in (Z_n \res \h_T(y_{k_0}))$, then let $Z_{n+1} = Z_n$ 
	(roughly speaking, (a) says $y_{k_0}$ is not yet ready to be handled, (b) says 
	$y_{k_0}$ is not relevant, and (c) says $y_{k_0}$ has already been handled). 
	Now suppose that $\h_T(y_{k_0}) \le \delta_n$, $F_n(y_{k_0}) = 1$, 
	and $y_{k_0} \notin (Z_n \res \h_T(y_{k_0}))$. 
	Since $F_n$ is a subtree function, 
	we can fix $y' \ge_T y_{k_0}$ with height $\delta_n$ such that $F_n(y') = 1$. 
	Note that $y' \notin Z_n \res \delta_n$. 
	As $T$ is normal, fix some $z \in T_\delta$ above $y'$. 
	Let $Z_{n+1} = Z_n \cup \{ z \}$. 
	By the inductive hypothesis and the fact that 
	$z \res \delta_n = y'$ is not in $Z_n \res \delta_n$, it easily follows that 
	$Z_{n+1}$ has $\les$-to-one drop-downs to $\delta_n$. 

	\emph{Case 2:} $i = 1$. 
	Consider $\mathcal U_{k_0}$, $\mathcal Y^{k_0}_{k_1}$, and $t_{k_2}$. 
	Let $l$ be the dimension of $\mathcal U_{k_0}$. 
	Applying the fact that $F_n$ fulfills $\mathcal U_{k_0}$, we can find 
	$\vec a = (a_0,\ldots,a_{l-1}) \in \mathcal Y^{k_0}_{k_1} \res \delta_n$ 
	which is disjoint from $(Z_n \cup t_{k_2}) \res \delta_n$. 
	Fix $\vec a^+ = (a_0^+,\ldots,a_{l-1}^+) \in \mathcal Y^{k_0}_{k_1}$ above $\vec a$. 
	Since $\vec a$ is injective, $\vec a^+$ has unique drop-downs to $\delta_n$. 
	Note that $\vec a^+$ is disjoint from $t_{k_2}$. 
	Define $Z_{n+1} = Z_n \cup \{ a_0^+,\ldots,a_{l-1}^+ \}$. 
	By the inductive hypothesis and the fact that 
	$\{ a_0^+,\ldots,a_{l-1}^+ \} \res \delta_n$ is disjoint from $Z_n \res \delta_n$, 
	it easily follows that 
	$Z_{n+1}$ has $\les$-to-one drop-downs to $\delta_{n+1}$.

	As a second step, we define $F_{n+1}$. 
	Note that by Lemma 4.5, $Z_{n+1} \res \delta_{n+1}$ has $\les$-to-one drop-downs 
	to $\delta_n$, and for all $y \in Z_{n+1} \res \delta_{n+1}$, 
	$F_n(y \res \delta_n) = 1$. 
	By the inductive hypothesis for $\delta_{n+1}$, we can find 
	an $\mathcal S$-splitting subtree function $F_{n+1}$ extending $F_n$ 
	with top level $\delta_{n+1}$ 
	such that for all $y \in Z_{n+1} \res \delta_{n+1}$, $F_{n+1}(y) = 1$.

	This completes the construction. 
	Define $F^+$ extending $\bigcup_n F_n$ so that for all $x \in T_{\delta}$, 
	$F^+(x) = 1$ if $x \in \bigcup_n Z_n$, and $F^+(x) = 0$ otherwise. 
	It is routine to check that $F^+$ is as required.
\end{proof}

We now move on to the second main result of the section, which is that 
we can extend an $\mathcal S$-splitting subtree function to fulfill a generalized promise 
assuming that the generalized promise is suitable for it in some way which we describe next. 

\begin{definition}[Suitable Generalized Promises]
	Suppose that $F$ is an $\mathcal S$-splitting subtree function with top level $\beta$ 
	and $\mathcal U$ is a generalized promise with dimension $l$ and 
	base level $\delta$, where $\delta > \beta$ is a limit ordinal. 
	We say that $\mathcal U$ is \emph{suitable for $F$} if for some 
	sequences $\langle \delta_n : n < \omega \rangle$ and 
	$\langle \vec b_n : n < \omega \rangle$:
	\begin{itemize}
		\item $\mathcal U(\delta)$ is a singleton $\{ \mathcal Y_0 \}$;
		\item $\langle \delta_n : n < \omega \rangle$ is 
		an increasing sequence 
		of countable ordinals with supremum $\delta$, where $\delta_0 > \beta$;
		\item $\langle \vec b_n : n < \omega \rangle$ is a sequence 
		of disjoint 
		tuples in $\mathcal Y_0$, where we 
		let $X_n$ be the elements in $\vec b_n$ for all $n < \omega$;
		\item for all $0 < k < \omega$, 
		$X_0 \cup \cdots \cup X_{k}$ has unique drop-downs to $\delta_{k-1}$;
		\item $X_0 \cup X_1$ has $(\le \! 2)$-to-one drop-downs to $\beta$ and 
		for all $1 < k < \omega$, $X_0 \cup \cdots \cup X_k$ has 
		$(\le \! 2)$-to-one drop-downs to $\delta_{k-2}$;
		\item for all $k > 0$, $(\bigcup_n X_n) \res \delta_{k-1} = 
		(X_0 \cup \cdots \cup X_k) \res \delta_{k-1}$;
		\item for all $x \in \bigcup_n X_n$, $F(x \res \beta) = 1$.
	\end{itemize}
\end{definition}

\begin{lemma}
	Suppose that $F$ is an $\mathcal S$-splitting subtree function with top level $\beta$, 
	$U$ is a promise on $T$ with base level greater than or equal to $\beta$ and root 
	$(x_0,\ldots,x_{m-1})$, and for all $i < m$, $F(x_i \res \beta) = 1$. 
	Then for some countable limit ordinal $\delta$ greater than the base level of $U$, 
	the function $\mathcal U$ with domain $[\delta,\omega_1)$ 
	defined by $\mathcal U(\beta) = \{ U \cap T_\beta^m \}$ is a generalized 
	promise which is suitable for $F$.
\end{lemma}

\begin{proof}
	Apply Lemma 4.6 to $U$ and then argue as in Lemma 4.8.
\end{proof}

\begin{proposition}[Adding Generalized Promises]
	Suppose that $F$ is an $\mathcal S$-splitting subtree function with top level $\beta$, 
	$\Gamma$ is a countable collection of generalized promises whose base levels are less than 
	or equal to $\beta$, and $F$ fulfills every member of $\Gamma$. 
	Assume that $\mathcal U$ is a generalized promise with base level $\delta > \beta$ 
	which is suitable for $F$. 
	Then there exists an $\mathcal S$-splitting subtree function $F^+$ 
	extending $F$ with top level $\delta$ 
	such that $F^+$ fulfills every member of 
	$\Gamma \cup \{ \mathcal U \}$.
\end{proposition}

\begin{proof}
	Let $\mathcal U(\delta) = \{ \mathcal Y_0 \}$. 
	Fix sequences $\langle \delta_n : n < \omega \rangle$, 
	$\langle \vec b_n : n < \omega \rangle$, and $\langle X_n : n < \omega \rangle$ 
	as in Definition 5.5.
	To help with the construction of $F^+$, we fix the following objects:
	\begin{itemize}
		\item an enumeration $\langle y_n : n < \omega \rangle$ of $T \res \delta$;
		\item an enumeration $\langle \mathcal U_n : n < \omega \rangle$ of $\Gamma$;
		\item for each $n < \omega$, an enumeration 
		$\langle \mathcal Y^n_k : k < \omega \rangle$ of $\mathcal U_n(\delta)$ 
		(possibly with repetitions);
		\item an enumeration $\langle t_n : n < \omega \rangle$ of all finite 
		subsets of $T_{\delta}$;
		\item a surjection $h : \omega \to 2 \times \omega^3$ such that every 
		element of the codomain has an infinite preimage.
	\end{itemize}

	We define by induction 
	an increasing sequence $\langle Z_n : n < \omega \rangle$ 
	of finite subsets of $T_{\delta}$ together with a sequence 
	$\langle F_n : n < \omega \rangle$ satisfying that for all $n < \omega$:
	\begin{enumerate}
		\item $Z_n$ has unique drop-downs to $\delta_n$;
		\item $Z_n \res \delta_n$ and $(\bigcup_m X_m) \res \delta_n$ are disjoint;
		\item $F_n$ is an $\mathcal S$-splitting subtree function extending $F$ 
		with top level $\delta_n$ such that $F_n(z \res \delta_n) = 1$ 
		for all $z \in Z_n \cup \bigcup_m X_m$;
		\item $F_n \subseteq F_{n+1}$.
	\end{enumerate}

	To begin, let $Z_0 = \emptyset$. 
	Since $X_0 \cup X_1$ has $(\le \! 2)$-to-one drop-downs to $\beta$, by Lemma 4.5 
	also $(X_0 \cup X_1) \res \delta_0$ has $(\le \! 2)$-to-one drop-downs to $\beta$. 
	We also know that for all $x \in X_0 \cup X_1$, $F(x \res \beta) = 1$. 
	By Proposition 5.4 (Extension), we can fix an $\mathcal S$-splitting subtree function 
	$F_0$ extending $F$ with top level $\delta_0$ which fulfills every member of $\Gamma$ and 
	satisfies that for all $x \in X_0 \cup X_1$, $F_0(x \res \delta_0) = 1$.
	Inductive hypotheses (1) and (2) are immediate, 
	and (3) follows from the fact that 
	$(\bigcup_m X_m) \res \delta_0 = (X_0 \cup X_1) \res \delta_0$.

	Now let $n < \omega$ be given and assume that $Z_n$ and $F_n$ are defined as required. 	
	Let $h(n) = (i,k_0,k_1,k_2)$. 
	We begin by defining $Z_{n+1}$ satisfying the following properties:
	\begin{itemize}
	\item $Z_{n+1}$ has unique drop-downs to $\delta_n$;
	\item $Z_{n+1} \res \delta_n$ and $(\bigcup_m X_m) \res \delta_n$ are disjoint;
	\item for all $x \in Z_{n+1}$, $F_n(x \res \delta_n) = 1$.
	\end{itemize}
	Note that these properties imply inductive hypotheses (1) and (2) for $n+1$.
	
	\emph{Case 1:} $i = 0$. 
	Consider $y_{k_0}$. 
	If either (a) $\h_T(y_{k_0}) > \delta_n$, (b) $\h_T(y_{k_0}) \le \delta_n$ 
	and $F_n(y_{k_0}) = 0$, or (c) $\h_T(y_{k_0}) \le \delta_n$ and 
	$y_{k_0} \in (Z_n \cup \bigcup_m X_m) \res \h_T(y_{k_0})$, then let $Z_{n+1} = Z_n$. 
	Now assume that $\h_T(y_{k_0}) \le \delta_n$, $F_n(y_{k_0}) = 1$, and 
	$y_{k_0}$ has nothing above it in $Z_n \cup (\bigcup_m X_m)$. 
	Since $F_n$ is a subtree function, we can 
	fix some $y' \ge_T y_{k_0}$ with height $\delta_n$ such that $F_n(y') = 1$. 
	As $T$ is normal, let $z \in T_\delta$ be above $y'$. 
	Define $Z_{n+1} = Z_n \cup \{ z \}$. 
	By the inductive hypothesis it easily follows that 
	$Z_{n+1}$ has unique drop-downs to $\delta_n$. 
	Since $Z_{n} \res \delta_n$ and $(\bigcup_m X_m) \res \delta_n$ are disjoint 
	and $y'$ is not in $(\bigcup_m X_m) \res \delta_{n}$, 
	it follows that $Z_{n+1} \res \delta_{n}$ and $(\bigcup_m X_m) \res \delta_{n}$ 
	are disjoint.

	\emph{Case 2:} $i = 1$. 
	Consider $\mathcal U_{k_0}$, 
	$\mathcal Y^{k_0}_{k_1}$, and $t_{k_2}$. 
	Let $l$ be the dimension of $\mathcal U_{k_0}$. 
	Applying the fact that $F_n$ fulfills $\mathcal U_{k_0}$, we can find 
	$\vec a = (a_0,\ldots,a_{l-1}) \in \mathcal Y^{k_0}_{k_1} \res \delta_n$ 
	which is disjoint from 
	$$
	(Z_n \cup t_{k_2} \cup X_0 \cup \cdots \cup X_{n+1}) \res \delta_n
	$$
	such that for all $i < l$, $F_n(a_i) = 1$. 
	Fix $\vec a^+ = (a_0^+,\ldots,a_{l-1}^+) \in \mathcal Y^{k_0}_{k_1}$ above $\vec a$. 
	Since $\vec a$ is injective, $\vec a^+$ has unique drop-downs to $\delta_n$. 
	Note that $\vec a^+$ is disjoint from $t_{k_2}$. 
	Define $Z_{n+1} = Z_n \cup \{ a_0^+,\ldots,a_{l-1}^+ \}$. 
	By the inductive hypothesis and the choice of $\vec a$, 
	it easily follows that $Z_{n+1}$ has unique drop-downs to $\delta_n$. 
	Also by the inductive hypothesis and the choice of $\vec a$, 
	$Z_{n+1} \res \delta_n$ and $(X_0 \cup \cdots \cup X_{n+1}) \res \delta_n$ 
	are disjoint. 
	But $(X_0 \cup \cdots \cup X_{n+1}) \res \delta_n = (\bigcup_m X_m) \res \delta_n$, 
	so $Z_{n+1} \res \delta_n$ and $(\bigcup_m X_m) \res \delta_n$ are disjoint.

	This completes the definition of $Z_{n+1}$ in both cases. 
	We know that 
	for all $x \in Z_{n+1} \cup X_0 \cup \cdots \cup X_{n+2}$, $F_n(x \res \delta_n) = 1$. 
	Also, $(\bigcup_m X_m) \res \delta_{n+1} = 
	(X_0 \cup \cdots \cup X_{n+2}) \res \delta_{n+1}$, 
	which by hypothesis and Lemma 4.5 has $(\le \! 2)$-to-one drop-downs to $\delta_{n}$. 
	Since $Z_{n+1}$ has unique drop-downs to $\delta_n$ and 
	$Z_{n+1} \res \delta_n$ and $(\bigcup_m X_m) \res \delta_n$ are disjoint, 
	it easily follows that $(Z_{n+1} \cup X_0 \cup \cdots \cup X_{n+2}) \res \delta_n$ 
	has $(\le \! 2)$-to-one drop-downs to $\delta_n$. 
	By Proposition 5.4 (Extension), fix an $\mathcal S$-splitting subtree function $F_{n+1}$ 
	extending $F_n$ with top level $\delta_{n+1}$ which fulfills every member of $\Gamma$ 
	and satisfies that $F_{n+1}(x \res \delta_{n+1}) = 1$ for all 
	$x \in Z_{n+1} \cup X_0 \cup \cdots \cup X_{n+2}$. 
	Since $(X_0 \cup \cdots \cup X_{n+2}) \res \delta_{n+1} = 
	(\bigcup_m X_m) \res \delta_{n+1}$, for all $x \in \bigcup_m X_m$, 
	$F_{n+1}(x \res \delta_{n+1}) = 1$.

	This completes the construction. 
	Now define $F^+$ extending $\bigcup_m F_m$ so that for all $x \in T_{\delta}$, 
	$F^+(x) = 1$ if $x \in \bigcup_m Z_m \cup \bigcup_m X_m$, and $F^+(x) = 0$ otherwise. 
	It is routine to check that $F^+$ is as required.
\end{proof}

\section{The Forcing Poset}

We are now ready to introduce our forcing poset for adding an uncountable 
downwards closed $\mathcal S$-splitting normal subtree of $T$.

\begin{definition}
	Let $\p(T,\mathcal S)$ be the forcing poset consisting of conditions which are pairs 
	$(F,\Gamma)$, where $F$ is an $\mathcal S$-splitting subtree function, 
	$\Gamma$ is a countable collection of generalized promises with base levels 
	less than or equal to the top level of $F$, 
	and $F$ fulfills every member of $\Gamma$. 
	Let $(G,\Sigma) \le (F,\Gamma)$ if $F \subseteq G$ and $\Gamma \subseteq \Sigma$.
\end{definition}

We abbreviate $\p(T,\mathcal S)$ as $\p$ for Sections 6-8 and 10.

If $p = (F,\Gamma) \in \p$ and $\beta$ is the top level of $F$, then we say that 
$\beta$ is the top level of $p$. 

\begin{proposition}[Extension]
	Let $(F,\Gamma) \in \p$ have top level $\beta$. 
	Let $\beta \le \gamma < \omega_1$ and let $X \subseteq T_\gamma$ be finite. 
	Assume that $X$ has $\les$-to-one drop-downs to $\beta$ 
	and for all $x \in X$, $F(x \res \beta) = 1$. 
	Then there exists some $G$ such that 
	$(G,\Gamma) \le (F,\Gamma)$, $(G,\Gamma)$ has top level $\gamma$, 
	and for all $x \in X$, $G(x) = 1$.
\end{proposition}

\begin{proof}
	Immediate from Proposition 5.4 (Extension).
\end{proof}

\begin{proposition}[Adding Generalized Promises]
	Let $(F,\Gamma) \in \p$ have top level $\beta$. 
	Suppose that $\mathcal U$ is a generalized promise with base level $\delta > \beta$ 
	which is suitable for $F$. 
	Then there exists some $G$ such that 
	$(G,\Gamma \cup \{ \mathcal U \}) \in \p$, 
	$(G,\Gamma \cup \{ \mathcal U \}) \le (F,\Gamma)$, and $G$ has top level $\delta$.
\end{proposition}

\begin{proof}
	Immediate from Proposition 5.7 (Adding Generalized Promises).
\end{proof}

The following result is the main tool for showing that $\p$ is totally proper.

\begin{proposition}[Consistent Extensions Into Dense Sets]
	Let $\lambda$ be a large enough regular cardinal. 
	Assume:
	\begin{itemize}
	\item $N \prec H(\lambda)$ is countable, and $T$ and $\p$ are in $N$;
	\item $(F,\Gamma) \in N \cap \p$ has top level $\beta$;
	\item $D$ is a dense open subset of $\p$ in $N$;
	\item $X \subseteq T_{N \cap \omega}$ is finite, 
	$X$ has $\les$-to-one drop-downs to $\beta$, 
	and for all $x \in X$, $F(x \res \beta) = 1$.
	\end{itemize}
	Then there exists a condition $(G,\Sigma) \le (F,\Gamma)$ in $N \cap D$ with 
	some top level $\gamma$ such that for all $x \in X$, $G(x \res \gamma) = 1$.
\end{proposition}

\begin{proof}
	Since $T$ is normal, we can 
	fix $\xi < N \cap \omega_1$ greater than or equal to 
	$\beta$ such that $X$ has unique drop-downs to $\xi$. 
	By Lemma 4.5, $X \res \xi$ has $\les$-to-one drop-downs to $\beta$. 
	Applying Proposition 6.2 (Extension) in $N$, find $F'$ in $N$ such that 
	$(F',\Gamma) \le (F,\Gamma)$, $(F',\Gamma)$ has top level $\xi$, 
	and for all $x \in X$, $F'(x \res \xi) = 1$. 
	Now to simplify notation, let us just assume that $\beta = \xi$ and 
	$F = F'$.

	Fix some injective tuple 
	$\vec a = (a_0,\ldots,a_{m-1})$ which lists the elements of $X$. 
	Suppose for a contradiction that for every $(G,\Sigma) \le (F,\Gamma)$ in $N \cap D$ 
	with some top level $\gamma$, 
	there exists $j < m$ such that $G(a_j \res \gamma) = 0$. 
	Let $x_i := a_i \res \beta$ for all $i < m$ and let 
	$\vec x = (x_0,\ldots,x_{m-1})$.

	Define $W$ as the set of all tuples $\vec b = (b_0,\ldots,b_{m-1})$ in 
	the derived tree $T_{\vec x}$ 
	satisfying that whenever $(G,\Sigma) \le (F,\Gamma)$ is in $D$ and has some 
	top level $\gamma$, where 
	the height of the elements of $\vec b$ is greater than or equal to $\gamma$, 
	then for some $j < m$, $G(b_j \res \gamma) = 0$. 
	By elementarity, $W \in N$. 
	Note that $W$ is downwards closed in $T_{\vec x}$.

	We claim that $W$ is uncountable. 
	If not, then by elementarity there exists some $\zeta \in N \cap \omega_1$ 
	greater than $\beta$ 
	such that every member of $W$ is in $(T \res \zeta)^m$. 
	We claim that $\vec a \res \zeta$ is in $W$, which is a contradiction. 
	Otherwise, there exists some $(G,\Sigma) \le (F,\Gamma)$ in $D$ with some top 
	level $\gamma$, where $\gamma \le \zeta$, such that for all $j < m$, 
	$G(a_j \res \gamma) = 1$. 
	Since $(a_0 \res \zeta,\ldots,a_{m-1} \res \zeta) \in N$, 
	by elementarity we may assume that $(G,\Sigma)$ is in $N$. 
	So $(G,\Sigma) \le (F,\Gamma)$ is in $N \cap D$ and for all $j < m$, 
	$G(a_j \res \gamma) = 1$, contradicting our assumptions.

	Apply Lemma 4.2 to fix a promise $U \subseteq W$ with the same root as $W$, 
	which by elementarity 
	we may assume is in $N$. 
	The root of $U$ is $\vec x$, so $\beta$ 
	is both the base level of $U$ and also the top level of $(F,\Gamma)$. 
	Also, $F(x_i) = 1$ for all $i < m$. 
	Applying Lemma 5.6, fix a 
	countable limit ordinal $\delta > \beta$ such that 
	the function $\mathcal U$ with domain $[\delta,\omega_1)$ 
	defined by $\mathcal U(\gamma) = \{ U \cap T_\gamma^m \}$ is a generalized 
	promise which is suitable for $F$. 
	By elementarity, we may assume that $\delta$ and $\mathcal U$ are in $N$. 
	Apply Proposition 6.3 (Adding Generalized Promises) in $N$ to find some $G$ in $N$ 
	such that $(G,\Gamma \cup \{ \mathcal U \}) \in \p$, 
	$(G,\Gamma \cup \{ \mathcal U \}) \le (F,\Gamma)$, and 
	$G$ has top level $\delta$. 

	Now fix $(H,\Lambda) \le (G,\Gamma \cup \{ \mathcal U \})$ in $N \cap D$. 
	Let $\gamma$ be the top level of $H$. 
	Since $H$ fulfills $\mathcal U$ and $\mathcal U(\gamma) = \{ U \cap T_\gamma^m \}$, 
	there exists a tuple $\vec b = (b_0,\ldots,b_{m-1}) \in U \cap T_\gamma^m$ 
	such that for all $i < m$, $H(b_i) = 1$. 
	Since $U \subseteq W$, $\vec b \in W$. 
	But the condition $(H,\Lambda)$ witnesses that $\vec b \notin W$, and 
	we have a contradiction.
\end{proof}

\begin{thm}[Existence of Total Master Conditions]
	Let $\lambda$ be a large enough regular cardinal. 
	Assume:
	\begin{itemize}
	\item $N \prec H(\lambda)$ is countable, and $T$ and $\p$ are in $N$;
	\item $(F,\Gamma) \in N \cap \p$ has top level $\beta$;
	\item $X \subseteq T_{N \cap \omega}$ is finite, $X$ has 
	$\les$-to-one drop-downs to $\beta$, 
	and for all $x \in X$, $F(x \res \beta) = 1$;
	\item $h$ is a function defined on the set of all generalized 
	promises on $T$ which lie in $N$ 
	such that for any $\mathcal U \in \dom(h)$ with some dimension $l$ 
	and base level $\gamma$, 
	$h(\mathcal U)$ is a generalized promise on $T$ with dimension $l$ and base level $\gamma$ 
	such that $\mathcal U \res [\gamma,N \cap \omega_1) = 
	h(\mathcal U) \res [\gamma,N \cap \omega_1)$.
	\end{itemize}
	Then there exists a condition $(G,\Sigma) \le (F,\Gamma)$ 
	with top level $N \cap \omega_1$ such that $\Sigma \subseteq N$, 
	$(G,\Sigma)$ is a total master condition for $N$, for all $x \in X$, $G(x) = 1$, 
	and for any $\mathcal U \in \Sigma$, $F$ fulfills $h(\mathcal U)$.
\end{thm}

For the purpose of proving that $\p$ is totally proper, we can ignore 
the third and fourth bullet points of Theorem 6.5 
(or let $X = \emptyset$ and let $h$ be the identity function). 
In Section 10, 
the third bullet point is used to show that $\p$ is $(< \! \omega_1)$-proper and 
the fourth bullet point 
is used to prove that $\p$ has the $\omega_2$-p.i.c.

\begin{proof}
	To help with the construction of $(G,\Sigma)$, 
	fix the following objects:
	\begin{itemize}
	\item an enumeration $\langle y_n : n < \omega \rangle$ of $T \res (N \cap \omega_1)$;
	\item an enumeration $\langle D_n : n < \omega \rangle$ of all dense open 
	subsets of $\p$ which lie in $N$;
	\item an enumeration 
	$\langle t_n : n < \omega \rangle$ of all finite subsets 
	of $T_{N \cap \omega_1}$;
	\item an enumeration $\langle \mathcal U_n : n < \omega \rangle$ 
	of all generalized promises on $T$ which lie in $N$;
	\item for every $n < \omega$, enumerations 
	$\langle \mathcal Y^n_k : k < \omega \rangle$ and 
	$\langle \mathcal Z^n_k : k < \omega \rangle$ of $\mathcal U_n(N \cap \omega_1)$ 
	and $h(\mathcal U_n)(N \cap \omega_1)$ respectively (possibly with repetitions);
	\item a surjection $g : \omega \to 3 \times \omega^3$ 
	such that every member of the codomain has an infinite preimage.
	\end{itemize}

	We define by induction sequences 
	$\langle (F_n,\Gamma_n) : n < \omega \rangle$, 
	$\langle \gamma_n : n < \omega \rangle$, 
	and $\langle Z_n : n < \omega \rangle$ satisfying the following properties:
	\begin{enumerate}
	\item $(F_0,\Gamma_0) = (F,\Gamma)$, $\gamma_0 = \beta$, and $Z_0 = X$;
	\item for all $n < \omega$, $(F_n,\Gamma_n) \in N \cap \p$ has 
	top level $\gamma_n$ and $(F_{n+1},\Gamma_{n+1}) \le (F_n,\Gamma_n)$;
	\item $\langle Z_n : n < \omega \rangle$ is an increasing sequence of  
	finite subsets of $T_{N \cap \omega_1}$;
	\item for all $n < \omega$, $Z_n$ has $\les$-to-one drop-downs to $\gamma_n$, and 
	for all $x \in Z_n$, $F_n(x \res \gamma_n) = 1$.
	\end{enumerate}

	Let $(F_0,\Gamma_0) = (F,\Gamma)$, $\gamma_0 = \beta$, 
	and $Z_0 = X$. 
	Now let $n < \omega$ and assume that $(F_n,\Gamma_n)$, $\gamma_n$, 
	and $Z_n$ are defined and satisfy the required properties. 
	Let $g(n) = (i,k_0,k_1,k_2)$.

	\underline{Case 1:} $i = 0$. 
	We consider the element $y_{k_0} \in T \res (N \cap \omega_1)$. 
	If either (a) $\h_T(y_{k_0}) > \gamma_n$, (b) $\h_T(y_{k_0}) \le \gamma_n$ 
	and $F_n(y_{k_0}) = 0$, or (c) $\h_T(y_{k_0}) \le \gamma_n$ and 
	$y_{k_0} \in (Z_n \res \h_T(y_{k_0}))$, then let $Z_{n+1} = Z_n$.
	Assume that $\h_T(y_{k_0}) \le \gamma_n$, $F_n(y_{k_0}) = 1$, and $y_{k_0}$ is not 
	below a member of $Z_n$. 
	Since $F_n$ is an $\mathcal S$-splitting subtree function, we can fix some 
	$y' \in T_{\gamma_n}$ with $y_{k_0} \le_T y'$ such that $F_n(y') = 1$. 
	As $T$ is normal, fix $z \in T_{N \cap \omega_1}$ above $y'$. 
	Define $Z_{n+1} = Z_n \cup \{ z \}$. 
	By the inductive hypothesis, $Z_{n+1}$ 
	has $\les$-to-one drop-downs to $\gamma_n$. 
	Now in either case, 
	let $(F_{n+1},\Gamma_{n+1}) = (F_n,\Gamma_n)$ and $\gamma_{n+1} = \gamma_n$.
	
	\underline{Case 2:} $i = 1$. 
	Applying Proposition 6.4 (Consistent Extensions Into Dense Sets), 
	fix a condition $(F_{n+1},\Gamma_{n+1}) \le (F_n,\Gamma_n)$ in 
	$N \cap D_{k_0}$ with some top level $\gamma_{n+1}$ such that 
	for all $x \in Z_{n}$, $F_{n+1}(x \res \gamma_{n+1}) = 1$. 
	Note that $Z_{n}$ has $\les$-to-one drop-downs to $\gamma_{n+1}$ 
	by Lemma 4.5. 
	Define $Z_{n+1} = Z_n$.

	\underline{Case 3:} $i = 2$. 
	If $\mathcal U_{k_0} \notin \Gamma_n$, then let 
	$(F_{n+1},\Gamma_{n+1}) = (F_n,\Gamma_n)$, 
	$\gamma_{n+1} = \gamma_n$, and $Z_{n+1} = Z_n$.  
	Assume that $\mathcal U_{k_0} \in \Gamma_n$. 
	Let $l$ be the dimension of $\mathcal U_{k_0}$. 
	Since $(F_n,\Gamma_n)$ is a condition, $F_n$ fulfills $\mathcal U_{k_0}$. 
	As $\mathcal Y^{k_0}_{k_1} \res \gamma_n \in \mathcal U_{k_0}(\gamma_n)$, 
	we can find a tuple $\vec d = (d_0,\ldots,d_{l-1})$ in 
	$\mathcal Y^{k_0}_{k_1} \res \gamma_n$ which is disjoint from 
	$(t_{k_2} \cup Z_n) \res \gamma_n$ and satisfies that 
	for all $i < l$, $F_n(d_i) = 1$. 
	Now $\mathcal U_{k_0} \res [\gamma_n,N \cap \omega_1) = 
	h(\mathcal U_{k_0}) \res [\gamma_n,N \cap \omega_1)$. 
	So $\mathcal Z^{k_0}_{k_1} \res \gamma_n \in h(\mathcal U_{k_0})(\gamma_n) = 
	\mathcal U_{k_0}(\gamma_n)$. 
	Again since $F_n$ fulfills $\mathcal U_{k_0}$, 
	we can find a tuple 
	$\vec e = (e_0,\ldots,e_{l-1})$ in 
	$\mathcal Z^{k_0}_{k_1} \res \gamma_n$ which is disjoint from both $\vec d$ and  
	$(t_{k_2} \cup Z_n) \res \gamma_n$ and satisfies that 
	for all $i < l$, $F_n(e_i) = 1$. 

	Fix tuples $\vec d^+ = (d_0^+,\ldots,d_{l-1}^+)$ 
	and $\vec e^+ = \{ e_0^+,\ldots,e_{l-1}^+ \}$ 
	which are above $\vec d$ and $\vec e$ and in $\mathcal Y^{k_0}_{k_1}$ 
	and $\mathcal Z^{k_0}_{k_1}$ respectively. 
	Observe that $\{ d_0^+,\ldots,d_{l-1}^+,e_0^+,\ldots,e_{l-1}^+ \}$ has 
	unique drop-downs to $\gamma_n$. 
	Define $Z_{n+1} = Z_n \cup \{ d_0^+,\ldots,d_{l-1}^+,e_0^+,\ldots,e_{l-1}^+ \}$. 
	Note that by the inductive hypothesis and the choice of $\vec d$ and $\vec e$,  
	$Z_{n+1}$ has $\les$-to-one drop-downs to $\gamma_n$. 
	Finally, let $(F_{n+1},\Gamma_{n+1}) = (F_n,\Gamma_n)$ 
	and $\gamma_{n+1} = \gamma_n$.

	This completes the induction. 
	Define a condition $(G,\Sigma)$ as follows. 
	Let $\Sigma = \bigcup_n \Gamma_n$. 
	Define $G$ which extends $\bigcup_n F_n$ so that for all $x \in T_{N \cap \omega_1}$, 
	$G(x) = 1$ if $x \in \bigcup_n Z_n$, and $G(x) = 0$ otherwise.
	We leave it to the reader to check that this works.
\end{proof}

\begin{corollary}
	The forcing poset $\p$ is totally proper.
\end{corollary}

It is routine to show that whenever 
$G$ is a generic filter on $\p$, then in $V[G]$ the set 
$$
U = \{ x \in T : \exists (F,\Gamma) \in G \ F(x) = 1 \}
$$
is an uncountable downwards closed subtree of $T$ which is normal and satisfies that 
for all $x \in U$, $|\Imm_U(x)| \in \mathcal S$.

The following lemma is easy.

\begin{lemma}
	Suppose that $(F,\Gamma)$ and $(F,\Sigma)$ are conditions in $\p$. 
	Then $(F,\Gamma \cup \Sigma) \in \p$ is a lower bound of both 
	$(F,\Gamma)$ and $(F,\Sigma)$.
\end{lemma}

\begin{corollary}
	Assuming \textsf{CH}, $\p$ is $\omega_2$-c.c.
\end{corollary}

\begin{proof}
	If \textsf{CH} holds, then there are only $\omega_1$-many 
	$\mathcal S$-splitting subtree functions.
\end{proof}

\section{Not Adding Small-Splitting Subtrees}

Recall that $T$ is a fixed normal Aronszajn tree, $\mathcal S \subseteq \omega \setminus 2$, 
for all $x \in T$, $|\Imm_T(x)| \ge \sup(\mathcal S)$, and $\p = \p(T,\mathcal S)$. 
In Section 11 we prove that for any $2 < n < \omega$, it is consistent that 
every $(\ge \! n)$-splitting normal Aronszajn tree contains an uncountable downwards closed 
$n$-splitting normal subtree, but there exists an infinitely splitting normal Aronszajn tree 
which contains no uncountable downwards closed $(< \! n)$-splitting subtree. 
For this purpose, we need the following theorem.

\begin{thm}
	Let $1 < n \le s$. 
	Suppose that $S$ is an $\omega_1$-tree which does not 
	contain an uncountable downwards closed $(< \! n)$-splitting subtree. 
	Then $\p$ forces that $S$ does not contain an uncountable 
	downwards closed $(< \! n)$-splitting subtree. 
\end{thm}

Applying the above theorem for $n = 2$, we get the following corollary.

\begin{corollary}
	The forcing poset $\p$ preserves Aronszajn trees. 
	In particular, $\p$ forces that $T$ is Aronszajn.
\end{corollary}

The next lemma is an easy consequence of the existence of total master conditions.

\begin{lemma}
	Let $S$ be an $\omega_1$-tree. 
	Let $(F,\Gamma) \in \p$ have top level $\beta$ and suppose that 
	$(F,\Gamma)$ forces that $\dot W$ is a subset of $S$. 
	Then there exists a club $D \subseteq \omega_1$ consisting of limit ordinals 
	greater than $\beta$ such that for all $\delta \in D$, for any 
	finite set $X \subseteq T_\delta$ with $\les$-to-one drop-downs to $\beta$ 
	and satisfying that for all $x \in X$, $F(x \res \beta) = 1$, there exists 
	some $(G,\Sigma) \le (F,\Gamma)$ with top level $\delta$ such that for all 
	$x \in X$, $G(x) = 1$, and $(G,\Sigma)$ decides $\dot W \cap (S \res \delta)$.
\end{lemma}

\begin{proof}
	Fix some large enough regular cardinal $\lambda$. 
	Let $\langle N_\alpha : \alpha < \omega_1 \rangle$ be an increasing and 
	continuous sequence of countable elementary substructures of $H(\lambda)$
	such that $T$, $\p$, $(F,\Gamma)$, $S$, and $\dot W$ are members of $N_0$ 
	and $N_\alpha \in N_{\alpha+1}$ for all $\alpha < \omega_1$. 
	Let $D = \{ \delta < \omega_1 : N_\delta \cap \omega_1 = \delta \}$.
	Then $D$ is a club set of limit ordinals greater than $\beta$.

	Consider any $\delta \in D$. 
	Let $X \subseteq T_\delta$ be finite with $\les$-to-one drop-downs to $\beta$ 
	and satisfying that for all $x \in X$, $F(x \res \beta) = 1$. 
	By Theorem 6.5 (Existence of Total Master Conditions), we can find 
	a total master condition $(G,\Sigma) \le (F,\Gamma)$ for $N_\delta$ 
	with top level $N_\delta \cap \omega_1 = \delta$ 
	such that for all $x \in X$, $G(x) = 1$. 
	Since $S \res \delta \subseteq N_\delta$ and 
	$(G,\Sigma)$ is a total master condition for $N_\delta$, 
	$(G,\Sigma)$ decides 
	$\dot W \cap (S \res \delta)$.
\end{proof}

\begin{lemma}
	Let $1 < n \le s$. 
	Assume that $S$ is an $\omega_1$-tree which does not 
	contain an uncountable downwards closed $(< \! n)$-splitting subtree. 
	Let $(F,\Gamma) \in \p$ have top level $\beta$ and suppose that 
	$(F,\Gamma)$ forces that $\dot W$ is a downwards closed 
	$(< \! n)$-splitting subtree of $S$.

	Let $\lambda$ be a large enough regular cardinal. 
	Assume:
	\begin{itemize}
	\item $N \prec H(\lambda)$ is countable, and $T$, $\p$, $(F,\Gamma)$, 
	$S$, and $\dot W$ are in $N$;
	\item $c \in S_{N \cap \omega_1}$;
	\item $Y$ is a finite subset of $T_{N \cap \omega_1}$ 
	with $\les$-to-one drop-downs to $\beta$ such that for all 
	$y \in Y$, $F(y \res \beta) = 1$.
	\end{itemize}
	Then there exists a condition $(G,\Sigma) \le (F,\Gamma)$ in $N$ with 
	some top level $\gamma$ such that for all $y \in Y$, $G(y \res \gamma) = 1$ 
	and $(G,\Sigma) \Vdash c \notin \dot W$.
\end{lemma}

\begin{proof}
	As in the first paragraph of the proof of Proposition 6.4, 
	without loss of generality we may assume that $Y$ has unique drop-downs to $\beta$. 
	Fix a club $D \subseteq \omega_1$ as described in Lemma 7.3. 
	By elementarity, we may assume that $D \in N$. 
	So $N \cap \omega_1$ is a limit point of $D$. 
	Suppose for a contradiction that for all $(G,\Sigma) \le (F,\Gamma)$ in $N$, 
	if $\gamma$ is the top level of $(G,\Sigma)$ and for all $y \in Y$,  
	$G(y \res \gamma) = 1$, then it is not the case that 
	$(G,\Sigma) \Vdash c \notin \dot W$. 
	Let $\vec a = (a_0,\ldots,a_{l-1})$ be an injective tuple which enumerates $Y$. 
	Let $x_i = a_i \res \beta$ for all $i < l$.

	Define $B$ as the set of all $b \in S$ satisfying that 
	$c \res \beta <_S b$, the height $\delta$ of $b$ is in $D$, and 
	there exists a tuple 
	$\vec y = (y_0,\ldots,y_{l-1}) \in T_{\delta}^l$ 
	such that $x_i <_T y_i$ for all $i < l$, and 
	whenever $(G,\Sigma) \le (F,\Gamma)$ 
	has some top level $\gamma \le \delta$ 
	and for all $i < l$, $G(y_i \res \gamma) = 1$, then 
	it is not the case that $(G,\Sigma) \Vdash b \notin \dot W$.
	
	Note the following facts:
	\begin{enumerate}
	\item $B \in N$ by elementarity;
	\item since $\dot W$ is forced to be downwards closed, 
	for all $\zeta \in D$ with $\beta < \zeta < N \cap \omega_1$, 
	$c \res \zeta \in B$ as witnessed by $\vec a \res \zeta$;
	\item similarly, if $b \in B$ as witnessed by $\vec y$, 
	and $\zeta \in D$ is such that $\beta < \zeta \le \h_T(b)$, 
	then $b \res \zeta \in B$ as witnessed by $\vec y \res \zeta$.
	\end{enumerate}
	Since $N \cap \omega_1$ is a limit point of $D$, 
	(1) and (2) imply that $B$ is uncountable.

	Let $\bar B$ be the set of $y \in S$ for which there exists some $b \in B$ 
	with $y <_{S} b$. 
	Then $\bar B$ is an uncountable downwards closed subtree of $S$. 
	Since $S$ does not contain an uncountable downwards closed $(< \! n)$-splitting subtree, 
	$\bar B$ is not $(< \! n)$-splitting. 
	So we can fix some $d \in \bar B$ which has at least $n$-many 
	immediate successors in $\bar B$. 
	Fix distinct immediate successors 
	$d_0,\ldots,d_{n-1}$ of $d$ in $\bar B$. 
	For each $i < n$, choose some $e_i \in B$ with $d_i <_{S} e_i$. 
	By property (3) in the previous paragraph and dropping down to the lowest height 
	of the elements $e_0,\ldots,e_{n-1}$, 
	we may assume without loss of generality 
	that all of the $e_i$'s have the same height $\delta$. 
	For each $i < n$, 
	pick a tuple $\vec y^i = (y_0^i,\ldots,y_{l-1}^i) \in T_{\delta}^l$ 
	which witnesses that $e_i \in B$.

	Now for all $i < n$ and $j < l$, 
	$y_j^i \res \beta = x_j$, and in particular, $F(y_j^i \res \beta) = 1$. 
	Define $X = \{ y_j^i : i < n, \ j < l \}$. 
	Then every member of $T_\beta$ either has nothing above it in $X$, or is equal 
	to $x_j$ for some $j < l$ and has exactly the elements $y_j^0,\ldots,y_j^{n-1}$ 
	above it in $X$. 
	Since $n \le s$, it follows that 
	$X$ has $\les$-to-one drop-downs to $\beta$. 
	As $\delta \in D$, we can fix some $(G,\Sigma) \le (F,\Gamma)$ with 
	top level $\delta$ such that for all $i < n$ and $j < l$, $G(y_j^i) = 1$ 
	and $(G,\Sigma)$ decides $\dot W \cap (S \res \delta)$. 
	Then for each $i < n$, since $\vec y^i$ witnesses that $e_i \in B$, 
	it is not the case that $(G,\Sigma) \Vdash e_i \notin \dot W$. 
	But $(G,\Sigma)$ forces that $\dot W$ is $(< \! n)$-splitting and $(G,\Sigma)$  
	decides $\dot W \cap S_{\h_{S}(d)+1}$. 
	So there must exist some $i < n$ 
	such that $(G,\Sigma) \Vdash d_i \notin \dot W$. 
	As $\dot W$ is forced to be downwards closed and $d_i <_{S} e_i$, 
	$(G,\Sigma) \Vdash e_i \notin \dot W$, which is a contradiction.
\end{proof}

\begin{proof}[Proof of Theorem 7.1]
	Let $1 < n \le s$ and suppose that $S$ is an $\omega_1$-tree which does not 
	contain an uncountable downwards closed $(< \! n)$-splitting subtree. 
	Consider a condition $(F,\Gamma) \in \p$ with some top level $\beta$ and assume 
	that $(F,\Gamma)$ forces that $\dot W$ is a downwards closed 
	$(< \! n)$-splitting subtree of $S$. 
	We find an extension of $(F,\Gamma)$ which forces that $\dot W$ is countable.

	Fix a large enough regular cardinal $\lambda$ and a countable $N \prec H(\lambda)$ 
	such that $T$, $\p$, $(F,\Gamma)$, $S$, and $\dot W$ are in $N$. 
	Now proceed as in the proof of Theorem 6.5 (Existence of Total Master Conditions) 
	to build a descending sequence $\langle (F_n,\Gamma_n) : n < \omega \rangle$ of 
	conditions in $N$ below $(F,\Gamma)$ and a lower bound 
	$(G,\Sigma)$ which is a total master condition for $N$, with the following adjustment. 
	Modify the bookkeeping so that for every $c \in S_{N \cap \omega_1}$, 
	there exists some $n < \omega$ such that $(F_{n+1},\Gamma_{n+1})$ forces that 
	$c \notin \dot W$. 
	This is possible by Lemma 7.4. 
	As $W$ is forced to be downwards closed, 
	clearly $(G,\Sigma)$ forces that 
	$\dot W$ is a subset of $S \res (N \cap \omega_1)$, and hence is countable.
\end{proof}

\section{Preserving Suslin Trees}

In some of our consistency results of Sections 9 and 11, we will need to preserve 
certain Suslin trees after forcing with $\p$. 
This is possible by the following theorem. 

\begin{thm}
	Suppose that $S$ is a normal Suslin tree such that $\Vdash_S T \ \text{is Aronszajn}$. 
	Then $\p$ forces that $S$ is Suslin.
\end{thm}

The main tool for proving Theorem 8.1 is the following lemma.

\begin{lemma}
	Suppose that $S$ is a normal Suslin tree such that $\Vdash_S T \ \text{is Aronszajn}$. 
	Let $(F,\Gamma) \in \p$ have top level $\beta$ and suppose that 
	$(F,\Gamma)$ forces that $\dot E$ is a dense open subset of $S$.
	
	Let $\lambda$ be a large enough regular cardinal. 
	Assume:
	\begin{itemize}
	\item $N \prec H(\lambda)$ is countable, and $T$, $\p$, $(F,\Gamma)$, 
	$S$, and $\dot E$ are in $N$;
	\item $c \in S_{N \cap \omega_1}$;
	\item $Y$ is a finite subset of $T_{N \cap \omega_1}$ 
	with $\les$-to-one drop-downs to $\beta$ such that for all 
	$y \in Y$, $F(y \res \beta) = 1$.
	\end{itemize}
	Then there exists a condition $(G,\Sigma) \le (F,\Gamma)$ in $N$ with 
	some top level $\gamma$ such that for all $y \in Y$, $G(y \res \gamma) = 1$ 
	and $(G,\Sigma) \Vdash c \in \dot E$.
\end{lemma}

\begin{proof}
	As in the first paragraph of the proof of 
	Proposition 6.4 (Consistent Extensions Into Dense Sets), 
	without loss of generality we may assume that $Y$ has unique drop-downs to $\beta$. 
	Suppose for a contradiction that for all $(G,\Sigma) \le (F,\Gamma)$ in $N$, 
	if $\gamma$ is the top level of $(G,\Sigma)$ and for all $y \in Y$,  
	$G(y \res \gamma) = 1$, then it is not the case that 
	$(G,\Sigma) \Vdash c \in \dot E$. 
	Fix an injective tuple $(a_0,\ldots,a_{l-1})$ which enumerates $Y$. 
	Define $x_i = a_i \res \beta$ for all $i < l$ and 
	let $\vec x = (x_0,\ldots,x_{l-1})$.

	Define $B$ as the set of all $b \in S$ satisfying that 
	$c \res \beta <_{S} b$, and letting $\delta = \h_S(b)$, there exists some 
	$\vec y = (y_0,\ldots,y_{l-1}) \in T_{\vec x} \cap T_\delta^l$ 
	such that whenever $(G,\Sigma) \le (F,\Gamma)$ 
	has some top level $\gamma \le \delta$ 
	and for all $i < l$, $G(y_i \res \gamma) = 1$, 
	then it is not the case that $(G,\Sigma) \Vdash b \in \dot E$.
	
	Note the following facts:
	\begin{enumerate}
	\item $B \in N$ by elementarity;
	\item since $\dot E$ is forced to be open, 
	for all $\zeta$ with $\beta \le \zeta < N \cap \omega_1$, 
	$c \res \zeta \in B$ as witnessed by $\vec a \res \zeta$;
	\item if $b \in B$ as witnessed by $\vec y$ 
	and $\zeta$ is such that $\beta \le \zeta \le \h_T(b)$, 
	then $b \res \zeta \in B$ as witnessed by $\vec y \res \zeta$.
	\end{enumerate}
	These properties 
	imply that $B$ is uncountable and downwards closed in $S_{c \res \beta}$.

	For any $e \in B$, let $\delta_e = \h_S(e)$ and 
	define $\mathcal Y_e$ to be the set of all 
	$\vec y = (y_0,\ldots,y_{l-1}) \in T_{\vec x} \cap T_{\delta_e}^l$ 
	witnessing that $e \in B$, which means that 
	whenever $(G,\Sigma) \le (F,\Gamma)$ has some top level $\gamma \le \delta_e$ 
	and for all $i < l$, $G(y_i \res \gamma) = 1$, then 
	it is not the case that $(G,\Sigma) \Vdash e \in \dot E$. 
	Note that by (3) above, for all $e$ and $f$ in $B$ such that $e <_S f$, 
	$\mathcal Y_f \res \delta_e \subseteq \mathcal Y_e$.

	Since $S_{c \res \beta}$ is a normal Suslin tree and $B$ is an uncountable downwards closed 
	subset of $S_{c \res \beta}$, we can fix some $d \ge_S c \res \beta$ 
	such that $S_d \subseteq B$. 
	Fix a generic filter $G_S$ on $S$ with $d \in G_S$. 
	So $G_S$ is a cofinal branch of $S$, and since $d \in G_S$, 
	$G_S \cap S_{c \res \beta} \subseteq B$. 
	Working in $V[G_S]$, define 
	$$
	W = \bigcup \{ \mathcal Y_e : e \in G_S \cap S_{c \res \beta} \}.
	$$
	So for all $\beta \le \gamma < \omega_1$, 
	$W \cap T_{\gamma}^l = \mathcal Y_{G_S(\gamma)}$. 
	By the last comment of the previous paragraph, $W$ is a downwards closed uncountable 
	subset of $T_{\vec x}$. 
	Using Lemma 4.2, fix a promise $U \subseteq W$ with root $\vec x$.

	By our assumption on $S$, $T$ is still Aronszajn in $V[G_S]$. 
	So we can apply Lemma 5.6 in $V[G_S]$ to find a countable limit ordinal $\delta > \beta$ 
	such that the function $\mathcal U$ with domain $[\delta,\omega_1)$ defined by 
	$\mathcal U(\gamma) = \{ U \cap T_\gamma^l \}$ is a generalized promise 
	which is suitable for $F$. 
	Now let $\dot W$, $\dot U$, $\dot \delta$, 
	and $\dot{\mathcal U}$ be $S$-names which $d$ forces 
	satisfy the above properties.

	Since $S$ is countably distributive, we can extend $d$ to $d'$ which decides 
	$\dot \delta$ and $\dot U \cap T_{\dot \delta}^l$ as 
	$\delta$ and $\mathcal Y_0$, and decides sequences 
	$\langle \delta_n : n < \omega \rangle$ and $\langle \vec b_n : n < \omega \rangle$ 
	which witness that $\dot{\mathcal U}$ is suitable for $F$. 
	Observe that by absoluteness these sequences satisfy the bullet points 
	of Definition 5.5 (Suitable Generalized Promises) in $V$.  
	Working in $V$ and using the countable distributivity of $S$, 
	define $\mathcal V$ with domain $[\delta,\omega_1)$ so that for all 
	$\delta \le \zeta < \omega_1$, $\mathcal V(\zeta)$ is the set 
	of all $\mathcal Y \subseteq T_\zeta^l$ such that for some $d^* \ge_S d'$, 
	$d^* \Vdash_S \mathcal Y = \dot U \cap T_{\zeta}^l$. 
	Note that $\mathcal V(\delta) = \{ \mathcal Y_0 \}$, and 
	each $\mathcal V(\zeta)$ is countable as $S$ is c.c.c.
	It is now routine to check 
	that $\mathcal V$ is a generalized promise in $V$ which is suitable for $F$.

	Applying Proposition 6.3 (Adding Generalized Promises), 
	fix some $G$ such that $(G,\Gamma \cup \{ \mathcal V \}) \le (F,\Gamma)$ and 
	$G$ has top level $\delta$. 
	Since $(F,\Gamma)$ forces that $\dot E$ is dense in $S$, 
	we can fix $(H,\Lambda) \le (G,\Gamma \cup \{ \mathcal V \})$ 
	which decides for some $e \ge_S d'$ 
	that $e \in \dot E$. 
	By Proposition 6.2 (Extension) and the fact that $\dot E$ is forced to be open in $S$, 
	we may assume without loss of generality that the height of $e$, 
	which is $\delta_e$, is the top level of $H$. 
	Now fix $f \ge_S e$ which decides $\dot U \cap T_{\delta_e}^l$ as $\mathcal Y$. 
	Then $\mathcal Y \in \mathcal V(\delta_e)$.

	Since $(H,\Lambda)$ fufills $\mathcal V$, there exists some 
	$\vec y = (y_0,\ldots,y_{l-1}) \in \mathcal Y$ 
	such that for all $i < l$, $H(y_i) = 1$. 
	Then $f$ forces that 
	$$
	\mathcal Y = \dot U \cap T_{\delta_e}^l \subseteq 
	\dot W \cap T_{\delta_e}^l = \mathcal Y_{\dot G_S(\delta_e)} = \mathcal Y_e.
	$$ 
	So $\vec y \in \mathcal Y_e$, and hence $\vec y$ witnesses that $e \in B$. 
	Since $H(y_i) = 1$ for all $i < l$, the definition of $B$ implies that 
	$(H,\Lambda)$ does not force that $e \in \dot E$, which is a contradiction.
\end{proof}

\begin{proof}[Proof of Theorem 8.1.]
	Let $S$ be a normal Suslin tree which forces that $T$ is Aronszajn. 
	To show that $\p$ forces that $S$ is Suslin, by Lemma 1.1(d) 
	it suffices to show that 
	whenever $(F,\Gamma)$ forces that $\dot E$ is a dense open subset of $S$, 
	then there exist $(G,\Sigma) \le (F,\Gamma)$ and $\delta < \omega_1$ 
	such that $(G,\Sigma) \Vdash S_\delta \subseteq \dot E$. 

	Fix a large enough regular cardinal $\lambda$ and a countable $N \prec H(\lambda)$ 
	such that $T$, $\p$, $(F,\Gamma)$, $S$, and $\dot E$ are in $N$. 
	Now proceed as in the proof of Theorem 6.5 (Existence of Total Master Conditions) 
	to build a descending sequence $\langle (F_n,\Gamma_n) : n < \omega \rangle$ of 
	conditions in $N$ below $(F,\Gamma)$ and a lower bound 
	$(G,\Sigma)$ which is a total master condition for $N$, with the following adjustment. 
	Modify the bookkeeping so that for every $c \in S_{N \cap \omega_1}$, 
	there exists some $n < \omega$ such that $(F_{n+1},\Gamma_{n+1})$ forces that $c \in \dot E$. 
	This is possible by Lemma 8.2. 
	Then $(G,\Sigma)$ forces that $S_{N \cap \omega_1} \subseteq \dot E$.
\end{proof}

\section{A Lindel\"{o}f Tree With Non-Lindel\"{o}f Square}

We are now prepared to solve Question 2.9, which asks whether 
the topological square of a Lindel\"{o}f tree is always Lindel\"{o}f.

\begin{thm}
	It is consistent that there exists a normal infinitely splitting $\omega_1$-tree $S$ 
	which is Suslin, and hence Lindel\"{o}f,  
	but the topological product $S \times S$ is not Lindel\"{o}f.
\end{thm}

Recall that in general the property of being Lindel\"{o}f is not preserved under products. 
For example, the Sorgenfrey line is Lindel\"{o}f but its topological square is not 
(\cite[Section 16]{willard}.

Before proving Theorem 9.1, we would like to clarify the relationship between 
the subspace topology on $S \otimes S$ induced by the product topology on $S \times S$ 
and the fine wedge topology on the tree $S \otimes S$.

\begin{proposition}
	Let $S$ be an infinitely splitting $\omega_1$-tree. 
	Then the subspace topology on $S \otimes S$ induced 
	by the product topology on $S \times S$ 
	is strictly finer than the fine wedge topology on $S \otimes S$.
\end{proposition}

\begin{proof}
	For showing that the subspace topology is finer than the fine wedge topology, 
	it suffices to show that every member of the subbase of 
	the fine wedge topology on $S \otimes S$ is open in the subspace topology. 
	Consider $U = (x,y) \up$ for some $(x,y) \in S \otimes S$. 
	Then $U = ((x \up) \times (y \up)) \cap (S \otimes S)$, which is open in 
	the subspace topology. 
	Now consider $V = S \otimes S \setminus (x,y) \up$ for some $(x,y) \in S \otimes S$. 
	Let $(a,b) \in V$, and we will find a set $W$ which is open in $S \times S$ 
	such that $(a,b) \in W \cap (S \otimes S) \subseteq V$. 
	By the definition of $V$, either $a \not \ge_T x$ or $b \not \ge_T y$. 
	By symmetry, assume the former. 
	Now check that $W = (T \setminus x \up) \times b \up$ works.

	It remains to show that the subspace topology is strictly finer than the 
	fine wedge topology. 
	Fix any distinct $x$ and $y$ of the same height in $S$ and let $z$ be an immediate 
	successor of $y$. 
	In the product topology on $S \times S$, the set 
	$V = (x \up) \times (y \up \setminus z \up)$ is open, 
	and hence $V \cap (S \otimes S)$ is open in the subspace topology. 
	Clearly, $(x,y) \in V$. 
	We claim that there does not exist any set $U \subseteq S \otimes S$ 
	which is a basic open set in the fine wedge topology 
	such that $(x,y) \in U \subseteq V$. 
	Otherwise, there exists a finite set $F$ of immediate successors of $(x,y)$ in $S \otimes S$ 
	such that $(x,y) \up \setminus F \up \subseteq V$. 
	Since $S$ is infinitely splitting, pick an immediate successor $d$ of $x$ such that 
	for all $(a,b) \in F$, $a \ne d$. 
	Then $(d,z) \notin F$. 
	So $(d,z) \in (x,y) \up \setminus F \up$, but certainly 
	$(d,z) \notin V$.
\end{proof}

Since any open set in the fine wedge topology on $S \otimes S$ 
is open in the subspace topology on $S \otimes S$ induced by the product topology, 
it easily follows that 
if the fine wedge topology on $S \otimes S$ is not Lindel\"{o}f, then neither is 
the subspace topology on $S \otimes S$. 
However, this is not sufficient for showing that the product topology on $S \times S$ 
is non-Lindel\"{o}f, since in general being Lindel\"{o}f is not preserved under subspaces. 
For example, $\omega_1+1$ with the order topology is Lindel\"{o}f, and in fact 
is compact, but its subspace $\omega_1$ is not Lindel\"{o}f.

However, the continuous image of any Lindel\"{o}f space is Lindel\"{o}f.

\begin{lemma}
	Let $S$ be an infinitely splitting $\omega_1$-tree. 
	Define $H : S \times S \to S \otimes S$ by 
	$$
	H(x,y) = (x \res \min(\h_S(x),\h_S(y)),y \res \min(\h_S(x),\h_S(y))),
	$$
	where we consider the domain of $H$ with the product topology and the 
	codomain of $H$ with the fine wedge topology. 
	Then $H$ is continuous and surjective.
\end{lemma}

\begin{proof}
	Clearly, $H$ is surjective. 
	For continuity, it suffices to show that whenever $U$ is a member of the subbase 
	of the fine wedge topology on $S \otimes S$, then $H^{-1}(U)$ is open in $S \times S$. 
	First, suppose that $U = (x,y) \up$, where $(x,y) \in S \otimes S$. 
	It is straightforward to show that 
	$H^{-1}(U) = x \up \times y \up$, which is open in $S \times S$. 
	Secondly, assume that $V = S \otimes S \setminus (x,y) \up$, 
	where $(x,y) \in S \otimes S$. 
	Then $H^{-1}(V) = S \times S \setminus H^{-1}((x,y) \up) = 
	S \times S \setminus (x \up \times y \up)$, which is also open in $S \times S$.
\end{proof}

\begin{corollary}
	Let $S$ be an infinitely splitting $\omega_1$-tree. 
	Suppose that $S \otimes S$ contains a finitely splitting uncountable 
	downwards closed subtree. 
	Then $S \times S$ is not Lindel\"{o}f.
\end{corollary}

\begin{proof}
	If $S \times S$ is Lindel\"{o}f, then by Lemma 9.3, so is the fine wedge topology 
	on $S \otimes S$. 
	By Theorem 2.7, $S \otimes S$ does not 
	contain a finitely splitting uncountable downwards closed subtree. 
\end{proof}

We now prove that 
it is consistent that there exists a normal infinitely splitting $\omega_1$-tree $S$ 
which is Suslin, and hence Lindel\"{o}f,  
but the topological product $S \times S$ is not Lindel\"{o}f.

\begin{proof}[Proof of Theorem 9.1.]
	Start with a model in which there exists a normal infinitely splitting Suslin 
	tree $S$ which has the \emph{unique branch property}, that is, forcing with $S$ 
	yields exactly one cofinal branch of $S$. 
	For example, assuming $\Diamond$ there exists a normal 
	infinitely splitting Suslin tree 
	which is \emph{free}, which means that any derived tree of $S$ is Suslin 
	(the rigid Suslin tree constructed under $\Diamond$ in 
	\cite[Chapter V]{devlinj} is free). 
	Fix distinct elements $x$ and $y$ on the same level of $S$. 
	We claim that if we force with $S$, then in the generic extension, the tree 
	$S_x \otimes S_y$ is still Aronszajn. 
	If not, then in $V^S$ a cofinal branch through $S_x \otimes S_y$ yields a cofinal 
	branch above $x$ and a cofinal branch above $y$. 
	Since $x$ and $y$ 
	are incomparable, these branches are distinct, which 
	contradicts the fact that $S$ has the unique branch property.
	
	Consider the forcing poset $\p = \p(S_x \otimes S_y,\omega \setminus 2)$. 
	Then $\p$ adds an 
	uncountable downwards closed finitely splitting subtree of $S_x \otimes S_y$. 
	Since $S$ forces that $S_x \otimes S_y$ is Aronszajn, 
	by Theorem 8.1 it follows that $\p$ forces that $S$ is Suslin. 
	Let $G$ be a generic filter on $\p$. 
	Then in $V[G]$, $S$ is Suslin, and by absoluteness it is still normal and 
	infinitely splitting. 
	In $V[G]$, there exists an uncountable downwards closed finitely splitting 
	subtree $U$ of $S_x \otimes S_y$. 
	Letting $U {\downarrow}$ denote the downward closure of $U$ in $S \otimes S$, 
	which just consists of $U$ together with the chain of elements of $S \otimes S$ 
	which are below $(x,y)$. 
	Then $U {\downarrow}$ is also finitely splitting. 
	By Corollary 9.4, $S \times S$ is not Lindel\"{o}f.
\end{proof}

\section{Additional Properties of the Forcing Poset}

In this section we briefly discuss some additional properties satisfied by the forcing 
poset $\p$ of Section 6 which imply that this forcing can be iterated with countable 
support in a well-behaved way. 
These properties are the $\omega_2$-p.i.c., $(< \! \omega_1)$-properness, 
and Dee-completeness.

\begin{definition}[{\cite[Chapter VIII, Section 2]{properimproper}}]
	A forcing poset $\q$ has the \emph{$\omega_2$-p.i.c.} if for any large enough 
	regular cardinal $\lambda > \omega_2$, assuming:
	\begin{itemize}
	\item $\alpha < \beta < \omega_2$;
	\item $N_\alpha$ and $N_\beta$ are countable elementary 
	substructures of $H(\lambda)$ with $\q \in N_\alpha \cap N_\beta$;
	\item $\alpha \in N_\alpha$, $\beta \in N_\beta$, $N_\alpha \cap \omega_2 \subseteq \beta$, 
	and $N_\alpha \cap \alpha = N_\beta \cap \beta$;
	\item there exists an isomorphism $h : (N_\alpha,\in) \to (N_\beta,\in)$ 
	which is the identity on 
	$N_\alpha \cap N_\beta$ and satisfies that $h(\alpha) = \beta$;
	\end{itemize}
	then for any $p \in N_\alpha \cap \q$, there exists a condition $q \le p, h(p)$ 
	which is both $(N_\alpha,\q)$-generic and $(N_\beta,\q)$-generic and satisfies 
	that for any $u \in N_\alpha \cap \q$ and for any $r \le q$, 
	there exists $s \le r$ such that ($s \le u$ iff $s \le h(u)$).
	\end{definition}

Our interest in this property is due to the following facts.

\begin{thm}[{\cite[Chapter VIII, Lemma 2.3]{properimproper}}]
	Assuming $\textsf{CH}$, any forcing poset which has the  
	$\omega_2$-p.i.c.\ is $\omega_2$-c.c.
\end{thm}

\begin{thm}
	Let $\alpha \le \omega_2$. 
	Suppose that $\langle \p_i, \dot \q_j : i \le \alpha, \ j < \alpha \rangle$ 
	is a countable support forcing iteration such that for all $i < \alpha$, 
	$\p_i$ forces that $\dot \q_i$ has the $\omega_2$-p.i.c. 
	If $\alpha < \omega_2$, then $\p_\alpha$ has the $\omega_2$-p.i.c., and if 
	$\alpha = \omega_2$ and \textsf{CH} holds, then $\p_{\omega_2}$ is $\omega_2$-c.c.
\end{thm}

\begin{proposition}
	The forcing poset $\p$ has the $\omega_2$-p.i.c.
\end{proposition}

\begin{proof}
	Let $\lambda$, $\alpha$, $\beta$, $N_\alpha$, $N_\beta$, and $h$ be as 
	in Definition 10.1. 
	Note that $\alpha$ and $\beta$ are uncountable and so 
	$N_\alpha \cap \omega_1 = N_\beta \cap \omega_1$. 
	It easily follows that whenever $\mathcal U \in N_\alpha$ is a  
	generalized promise on $T$ with dimension $l$ and base level $\gamma$, 
	then $h(\mathcal U)$ is a generalized promise on $T$ with 
	dimension $l$ and base level $\gamma$, and moreover, 
	$\mathcal U \res [\gamma,N_\alpha \cap \omega_1) = 
	h(\mathcal U) \res [\gamma,N_\alpha \cap \omega_1)$. 
	Consider $(F,\Gamma) \in N_\alpha \cap \p$. 
	By Theorem 6.5 (Existence of Total Master Conditions), 
	there exists $(G,\Sigma) \le (F,\Gamma)$ which is a total master 
	condition for $N_\alpha$, has top level $N_\alpha \cap \omega_1$, and satisfies 
	that $\Sigma \subseteq N$ and 
	for any $\mathcal U \in \Sigma$, $G$ fulfills $h(\mathcal U)$.

	We claim that $(G,h[\Sigma])$ is an extension of $h((F,\Gamma)) = (F,h[\Gamma])$ 
	which is a total master condition for $N_\beta$. 
	The fact that $(G,h[\Sigma])$ is a condition extending $(F,h[\Gamma])$ is clear. 
	To show that $(G,h[\Sigma])$ is a total master condition for $N_\beta$, it suffices 
	to show that if $v \in N_\alpha \cap \p$ and $(G,\Sigma) \le v$, then 
	$(G,h[\Sigma]) \le h(v)$. 
	For then if $E \in N_\beta$ is a dense open subset of $\p$ 
	and $v \in h^{-1}(E) \cap N_\alpha$ 
	satisfies that $(G,\Sigma) \le v$, where such $v$ exists since $(G,\Sigma)$ is a 
	total master condition for $N_\alpha$, 
	then $h(v) \in E \cap N_\beta$ and 
	$(G,h[\Sigma]) \le h(v)$. 
	So let $v = (H,\Lambda)$ be in $N_\alpha \cap \p$ such that $(G,\Sigma) \le v$. 
	Then $H \subseteq G$ and $\Lambda \subseteq \Sigma$. 
	Therefore, $h(v) = (H,h[\Lambda])$, $H \subseteq G$, and 
	$h[\Lambda] \subseteq h[\Sigma]$, so $(G,h[\Sigma]) \le h(v)$.

	By Lemma 6.7, $(G,\Sigma \cup h[\Sigma])$ is a condition which extends both 
	$(G,\Sigma)$ and $(G,h[\Sigma])$, and hence is a total master condition for 
	both $N_\alpha$ and $N_\beta$ which extends 
	$(F,\Gamma)$ and $h(F,\Gamma)$. 
	Now consider $u \in N_\alpha \cap \p$ and $r \le (G,\Sigma \cup h[\Sigma])$. 
	We find $s \le r$ such that $s \le u$ iff $s \le h(u)$. 
	Consider the dense open set $D$ in $N_\alpha$ of conditions which are either below $u$ 
	or incompatible with $u$. 
	Fix $v \in D \cap N_\alpha$ such that $(G,\Sigma) \le v$. 
	By the previous paragraph, $(G,h[\Sigma]) \le h(v)$. 
	As $r \le (G,\Sigma \cup h[\Sigma]) \le (G,\Sigma), (G,h[\Sigma])$, 
	it follows that $r \le v, h(v)$. 
	Case 1: $v \le u$. 
	Then $h(v) \le h(u)$. 
	So $r \le u, h(u)$. 
	Case 2: $v$ is incompatible with $u$. 
	Then $h(v)$ is incompatible with $h(u)$. 
	Since $r \le v, h(v)$, $r$ is incompatible with $u$ and $h(u)$, and in particular, 
	$r$ does not extend $u$ or $h(u)$.
\end{proof}

\begin{definition}[{\cite[Chapter V, Definition 3.1]{properimproper}}]
	For any $\alpha < \omega_1$, a forcing poset $\q$ is 
	\emph{$\alpha$-proper} if for any large enough regular cardinal $\lambda$, 
	if $\langle N_i : i \le \alpha \rangle$ is an increasing 
	and continuous sequence of countable elementary substructures of 
	$H(\lambda)$ with $\q \in N_0$ and for all $\beta < \alpha$, 
	$\langle N_i : i \le \beta \rangle \in N_{\beta+1}$, 
	then for all $p \in N_0 \cap \p$, 
	there exists $q \le p$ such that $q$ is $(N_\beta,\q)$-generic for 
	all $\beta \le \alpha$. 
	A forcing poset $\q$ is \emph{$(< \! \omega_1)$-proper} if it is $\alpha$-proper 
	for all $\alpha < \omega_1$.
\end{definition}

\begin{proposition}
	The forcing poset $\p$ is $(< \! \omega_1)$-proper.
\end{proposition}

\begin{proof}
	Let $\lambda$ be a large enough regular cardinal. 
	The goal is to prove by induction on $\alpha < \omega_1$ the following statement:
	Assuming
	\begin{itemize}
	\item $\langle N_i : i \le \alpha \rangle$ is an increasing and continuous sequence 
	of countable elementary substructures of $H(\lambda)$ with $\p \in N_0$;
	\item for all $\gamma < \alpha$, 
	$\langle N_i : i \le \gamma \rangle \in N_{\gamma+1}$;
	\item $(F,\Gamma) \in N_0 \cap \p$ has top level $\beta$;
	\item $X \subseteq T_{N_\alpha \cap \omega_1}$ is finite, $X$ has unique 
	drop-downs to $\beta$, and for all $x \in X$, $F(x \res \beta) = 1$;
	\end{itemize}
	then there exists $(G,\Sigma) \le (F,\Gamma)$ which is a total master condition 
	for $N_\gamma$ for all $\gamma \le \alpha$, has top level $N_\alpha \cap \omega_1$, 
	and satisfies that for all $x \in X$, $G(x) = 1$.
	
	The base case $\alpha = 0$ 
	is immediate from Theorem 6.5 (Existence of Total Master Conditions). 
	For the successor case, suppose that $\alpha = \gamma+1$ and the statement holds 
	for all ordinals less than $\alpha$. 
	Let $\langle N_i : i \le \gamma+1 \rangle$, $(F,\Gamma)$, $\beta$, and $X$ be as above. 
	By the inductive hypothesis applied to $X \res (N_\gamma \cap \omega_1)$, 
	we can find $(G,\Sigma) \le (F,\Gamma)$ which is a total 
	master condition for $N_\xi$ for all $\xi \le \gamma$, has top level 
	$N_\gamma \cap \omega_1$, and satisfies that for all $x \in X$, 
	$G(x \res (N_\gamma \cap \omega_1)) = 1$. 
	Now apply Theorem 6.5 (Existence of Total Master Conditions) to find 
	$(H,\Lambda) \le (G,\Sigma)$ which is a total master condition for 
	$N_{\gamma+1}$ with top level $N_{\gamma+1} \cap \omega_1$ 
	such that for all $x \in X$, $H(x) = 1$.
	
	Now assume that $\alpha$ is a limit ordinal and the result holds for all 
	$\gamma < \alpha$. 
	Let $\langle N_i : i \le \alpha \rangle$, $(F,\Gamma)$, $\beta$, and $X$ 
	be as above. 
	Fix an increasing sequence $\langle \alpha_n : n < \omega \rangle$ 
	of non-zero ordinals cofinal in $\alpha$. 
	Now repeat the argument of Theorem 6.5 (Existence of Total Master Conditions)
	to build a descending sequence of conditions 
	$\langle (F_n,\Gamma_n) : n < \omega \rangle$ in $N_\alpha$ below $(F,\Gamma)$, 
	with the following adjustment. 
	At a stage $n+1$ where we previously extended to get into a dense open set, 
	we instead apply the inductive hypothesis to extend to a condition with top level 
	$N_{\alpha_{n}} \cap \omega_1$ which is $(N_\gamma,\p)$-generic for all 
	$\gamma \le \alpha_n$. 
	We leave the details to the interested reader.
\end{proof}

In \cite[Chapter V, Section 5]{properimproper}, the notion of a 
\emph{simple $\omega_1$-completeness system $\mathbb D$} is defined, as well as 
the idea of a forcing poset $\q$ being \emph{$\mathbb D$-complete}. 
Define a forcing $\q$ to be \emph{Dee-complete} if there exists some 
simple $\omega_1$-completeness system $\mathbb D$ such that 
$\q$ is $\mathbb D$-complete.

\begin{thm}[{\cite[Chapter V, Theorem 7.1]{properimproper}}]
	Suppose that $\langle \p_i, \dot \q_j : i \le \alpha, \ j < \alpha \rangle$ 
	is a countable support forcing iteration such that for all $i < \alpha$, 
	$\p_i$ forces that $\dot \q_i$ is $(< \! \omega_1)$-proper and Dee-complete. 
	Then $\p_\alpha$ does not add reals.
\end{thm}

\begin{proposition}
	The forcing poset $\p$ is Dee-complete.
\end{proposition}

We omit the proof because it is essentially identical to the proof 
of \cite[Chapter V, Section 6]{properimproper} that the forcing for specializing 
an Aronszan tree with promises is Dee-complete 
(and that proof itself is quite trivial). 
We leave the matter to be explored by the interested reader.

\section{Consistency Results}

In this section, we complete the main consistency results of the article.

\begin{lemma}
	Let $2 \le n < \omega$. 
	Suppose that every normal infinitely splitting Aronszajn tree 
	contains an uncountable downwards closed $n$-splitting subtree. 
	Then every Aronszajn tree contains an uncountable downwards closed 
	$(\le \! n)$-splitting subtree.
\end{lemma}

\begin{proof}
	By Lemma 1.1(a), the assumption implies that every normal infinitely splitting Aronszajn 
	tree contains an uncountable downwards closed $n$-splitting normal subtree. 
	Let $T$ be an Aronszajn tree.
	We may assume that $T$ has a root, for otherwise we can restrict our attention to 
	$T_x$ for some $x$ of height $0$. 
	If $T$ is not Hausdorff, then by a standard shifting argument 
	find a Hausdorff tree $T^+$ 
	such that $T^+$ is equal to $T$ restricted to non-limit levels. 
	Note that if $T^+$ has an uncountable downwards closed $\lem$-splitting subtree, 
	then so does $T$. 
	So without loss of generality assume that $T = T^+$. 
	By Lemma 1.1(a), $T$ contains an uncountable downwards closed subtree $U$ such that 
	every element of $U$ has uncountably many elements of $U$ above it. 
	So again let us just assume that $T = U$. 
	Then $T$ is normal. 
	So by Lemma 1.1(b), there exists a club $C \subseteq \omega_1$ such that $T \res C$ 
	is infinitely splitting. 
	Note that $T \res C$ is normal. 
	Find an uncountable downwards closed normal subtree $W \subseteq T \res C$ 
	which is $m$-splitting. 
	By the normality of $W$, 
	it is easy to check that the downwards closure of $W$ in $T$ is 
	$(\le \! m)$-splitting, and we are done.
\end{proof}

\begin{thm}
	It is consistent that \textsf{CH} holds 
	and for all $2 \le n < \omega$, 
	every $(\ge \! n)$-splitting normal Aronszajn tree contains an 
	uncountable downwards closed 
	$n$-splitting normal subtree. 
	In particular, it is consistent that \textsf{CH} holds and 
	there do not exist any Lindel\"{o}f trees.
\end{thm}

\begin{proof}
	Start with a model satisfying \textsf{GCH}. 
	Define a countable support forcing iteration 
	$\langle \p_i, \dot \q_j : i \le \omega_2, \ j < \omega_2 \rangle$ 
	where we bookkeep 
	so that for every $2 \le n < \omega$ and every normal 
	$(\ge \! n)$-splitting normal Aronszajn tree $T$ in the final model, 
	there exists some $i < \omega_2$ such that 
	$\Vdash_{i} \dot \q_i = \p(T,\{n\})$. 
	By Proposition 10.4 and Theorems 10.2 and 10.3, 
	this bookkeeping is possible since each $\p_i$ has size $\omega_2$ and is $\omega_2$-c.c. 
	By Propositions 10.6 and 10.8 and Theorem 10.7, 
	$\p_{\omega_2}$ does not add reals, and in particular, 
	preserves $\omega_1$ and \textsf{CH}. 
	We claim that in $V^{\p_{\omega_2}}$ there do not exist any Lindel\"{o}f trees. 
	Start with an $\omega_1$-tree $T$. 
	If $T$ is not Aronszajn, then it is not Lindel\"{o}f by \cite[Lemma 4.2]{marun}. 
	So assume that $T$ is Aronszajn. 
	By Lemma 11.1, $T$ contains an uncountable downwards closed 
	$(\le \! 2)$-splitting subtree. 
	By Theorem 2.7, $T$ is not Lindel\"{o}f.
\end{proof}

Now we turn to consistency results involving the preservation of Suslin trees. 
We make use of the following forcing iteration theorem.

\begin{thm}[{\cite{AS2}}, {\cite{MIYAMOTO}}]
	Let $S$ be a Suslin tree. 
	Then the property of a forcing poset being proper and forcing that 
	$S$ is Suslin is preserved by any countable support forcing iteration.
\end{thm}

\begin{thm}[\cite{lindstrom}]
	Let $T$ be an Aronszajn tree which does not contain a Suslin subtree. 
	Then for any Suslin tree $S$, $\Vdash_S \text{$T$ is Aronszajn}$.
\end{thm}

Namely, if $x \in S$ forces that $T$ is not Aronszajn, then there 
exists a club $C \subseteq \omega_1$ and a strictly increasing function 
of $S_x \res C$ into $T$. 
The image of this function is a Suslin subtree of $T$.

\begin{lemma}
	Let $2 \le m < \omega$. 
	Suppose that every normal infinitely splitting Aronszajn tree 
	which does not contain a Suslin subtree 
	contains an uncountable downwards closed $m$-splitting subtree. 
	Then every Aronszajn tree which does not contain a Suslin subtree 
	contains an uncountable downwards closed 
	$(\le \! m)$-splitting subtree.
\end{lemma}

The proof is the same as for Lemma 11.1, using the fact that the downward closure 
of a Suslin subtree is Suslin.

\begin{thm}
	It is consistent that \textsf{CH} holds, there exists a Suslin tree, 
	and for every $\omega_1$-tree $T$, $T$ either contains a Suslin subtree 
	or is non-Lindel\"{o}f.
\end{thm}

\begin{proof}
	Start with a model satisfying \textsf{GCH} and there exists a Suslin tree. 
	Repeat the construction of Theorem 11.2 to build 
	$\langle \p_i, \dot \q_j : i \le \omega_2, \ j < \omega_2 \rangle$, 
	except that we obtain the conclusion only 
	for those normal Aronszajn trees which contain no Suslin subtree. 
	By Theorems 8.1 and 11.4, each forcing we iterate preserves all Suslin trees. 
	By Theorem 11.3, $\p_{\omega_2}$ preserves all Suslin trees, and in fact,  
	if $T$ is a normal Aronszajn tree which has no Suslin subtree 
	in the final model, then it does not have a Suslin subtree in any intermediate model. 
	Now repeat the same argument as in the end of the proof of 
	Theorem 11.2, but using Lemma 11.5 instead of Lemma 11.1
\end{proof}

Finally, we show that we can consistently separate different degrees of subtree splitting.

\begin{thm}
	Let $2 < n \le \omega$. 
	It is consistent that \textsf{CH} holds, 
	every normal $(\ge \! n)$-splitting Aronszajn tree contains an 
	uncountable downwards closed $n$-splitting normal subtree, and 
	there exists a normal Aronszajn tree which is infinitely splitting 
	but does not contain an uncountable downwards closed 
	$(< \! n)$-splitting subtree.
\end{thm}

\begin{proof}
	Start with a model satisfying \textsf{GCH} and there exists a normal 
	infinitely splitting Suslin tree $S$. 
	By \cite[Theorem 4.12]{marun}, $S$ does not contain an uncountable downwards 
	closed finitely splitting subtree. 
	In particular, $S$ does not contain an uncountable downwards closed 
	$(< \! n)$-splitting subtree. 
	Now define a countable support forcing iteration 
	$\langle \p_i, \dot \q_j : i \le \omega_2, \ j < \omega_2 \rangle$ as 
	in the proof of Theorem 11.2, except that we only force with 
	$\p(T,\{n\})$ for normal Aronszajn trees $T$ which are $(\ge \! n)$-splitting. 
	By Theorem 7.1, each such forcing preserves the fact that $S$ does not 
	contain an uncountable downwards closed 
	$(< \! n)$-splitting subtree, so by Theorem 3.1 so does $\p_{\omega_2}$.
\end{proof}

\textbf{Acknowledgements:} 
The author thanks Pedro Marun for providing 
helpful comments and corrections on an earlier draft of the article, and the referee
for writing a careful report.

\bigskip

\textbf{Funding:} The author acknowledges support from the Simons Foundation under 
the Travel Support for Mathematicians gift 631279.

\providecommand{\bysame}{\leavevmode\hbox to3em{\hrulefill}\thinspace}
\providecommand{\MR}{\relax\ifhmode\unskip\space\fi MR }
\providecommand{\MRhref}[2]{%
  \href{http://www.ams.org/mathscinet-getitem?mr=#1}{#2}
}
\providecommand{\href}[2]{#2}


\end{document}